\documentclass[12pt]{article}
\usepackage{amsfonts}
\usepackage{comment}
\usepackage{tabularx}
\usepackage[protrusion=true,expansion=true]{microtype}
\usepackage{enumerate}
\usepackage{bbm}
\usepackage[normalem]{ulem}
\usepackage{float}
\usepackage{tikz}
\usepackage{latexsym}
\usepackage{mathtools}
\usepackage[pdftitle={Minkowski content of Brownian cut points},
pdfauthor={Nina Holden, Greg Lawler, Xinyi Li, Xin Sun},
bookmarks=true,bookmarksopen=true,bookmarksopenlevel=3]{hyperref}

\usepackage{amsmath, amsthm, amssymb, color, graphicx, hyperref}
\usepackage[margin=1in]{geometry}
\DeclareGraphicsExtensions{.jpg,.pdf,.eps}
\graphicspath{{./Figures/}}

\newtheorem{thm}{Theorem}[section]
\newtheorem{proposition}[thm]{Proposition}
\newtheorem{prop}[thm]{Proposition}
\newtheorem{lemma}[thm]{Lemma}

\newtheorem{rmk}[thm]{Remark}
\newtheorem{remark}[thm]{Remark}

\newtheorem{observation}[thm]{Observation}

\newcommand{\ov}{\overline}

 \newcommand  \newset  {{\mathcal V}} 
\newcommand{\dist}{\mathrm{dist}}

\newcommand{\C}{\mathbb{C}}
\newcommand{\A}{\mathcal{A}}

\newcommand{\D}{\mathbb{D}}
\newcommand{\E}{\mathbb{E}}
\newcommand{\N}{\mathbb{N}}

\newcommand{\R}{\mathbb{R}}

\newcommand{\bbH}{\mathbb{H}}

\newcommand{\wt}{\widetilde}

\newcommand{\eps}{\varepsilon}

\def\E{\mathbb{E}}

\DeclareMathOperator{\SLE}{SLE}

\def\cU{\mathcal{U}}

\def\cP{\mathcal{P}}

\def\cD{\mathcal{D}}

\def\cB{\mathcal{B}}
\def\cA{\mathcal{A}}

\renewcommand{\Re}{\op{Re}}
\renewcommand{\Im}{\op{Im}}

\def\@rst #1 #2other{#1}
\newcommand\MR[1]{\relax\ifhmode\unskip\spacefactor3000 \space\fi
		\MRhref{\expandafter\@rst #1 other}{#1}}\newcommand{\MRhref}[2]{\href{http://www.ams.org/mathscinet-getitem?mr=#1}{MR#2}}

\def\MR#1{\href{http://www.ams.org/mathscinet-getitem?mr=#1}{MR#1}}

\newcommand{\aryb}{\begin{eqnarray*}}
\newcommand{\arye}{\end{eqnarray*}}
\def\alb#1\ale{\begin{align*}#1\end{align*}}
\newcommand{\eqb}{\begin{equation}}
\newcommand{\eqe}{\end{equation}}
\newcommand{\eqbn}{\begin{equation*}}
\newcommand{\eqen}{\end{equation*}}
\numberwithin{equation}{section}

\newcommand{\op}{\operatorname}

\newcommand{\rta}{\rightarrow}

\newcommand{\wh}{\widehat}

\newcommand{\bdy}{\partial}

\newcommand{\old}[1]{{}} 

\newcommand{\invprob} {{\bf Q}}
\newcommand {\eset}  {{\emptyset}}
\newcommand {\Prob} {{\bf {P}}}
\newcommand {\ball} {{\mathcal B}}
\newcommand{\e}  {{\bf e}}
\newcommand{\greencut}  {G^{\rm cut}}
\newcommand {\Cont}{{\rm Cont}}
\newcommand {\Vol}{{\rm Vol}}
\newcommand {\state} {{\mathcal X}}
\newcommand {\rev} {{R}}
\newcommand {\n}{{\bf n}}

\newcommand {\newstate}  {{\mathcal Y}}
\newcommand {\newstatetwo}  {{\mathcal Z}}
\newcommand {\doublestate} {{\mathcal J}}
\newcommand {\dyadic}  {{\mathcal D}}
\usepackage[affil-it]{authblk}

\title{Minkowski content of Brownian cut points}
\author[1]{Nina Holden\thanks{Partially supported by a fellowship from the Research Council of Norway and partially supported by Dr.\ Max R\"ossler, the Walter Haefner Foundation, and the ETH Z\"urich Foundation.}}
\author[2]{Gregory F. Lawler\thanks{Research supported by NSF grant DMS-1513036}}
\author[3]{Xinyi Li\thanks{Research supported by NSFC (No.\ 12071012) and the National Key R\&D Program of China (No.\ 2020YFA0712900)}}
\author[4]{Xin Sun\thanks{Supported by a Junior Fellow award from the Simons Foundation, and NSF Grants DMS-1811092 and DMS-2027986}}
\affil[1]{\small ETH Z\"urich}
\affil[2]{\small The University of Chicago}
\affil[3]{\small Beijing International Center for Mathematical Research, Peking University}
\affil[4]{\small University of Pennsylvania}

\date{\small \today}

\begin{document}

\maketitle
\begin{abstract}
Let $W(t)$, $0\leq t\leq T$, be a Brownian motion in $\mathbb{R}^d$, $d=2,3$.  We say that $x$ is a cut point for $W$ if $x=W(t)$ for some $t\in(0,T)$ such that $W [0,t) $ and $W (t,T]$ are disjoint.  
In this work, we prove that a.s.\ the Minkowski content of the set of cut points for $W$ exists and is finite and non-trivial. 

\bigskip

Soit $W(t)$, $0\leq t\leq T$ un mouvement brownien sur $\mathbb{R}^d$, $d=2,3$. On dit que $x$ est un point de coupure pour $W$ si $x=W(t)$ pour un certain $t\in(0,T)$ tel que $W [0,t) $ et $W (t,T]$ sont disjointes. Dans cet article, nous montrons que le contenu de Minkowski de l'ensemble des points de coupure pour $W$ existe p.s., et qu'il est p.s.\ fini et non trivial.

\end{abstract}

\section{Introduction}\label{sec:intro}
Let $W(t)$  be a Brownian motion taking values in $\R^d$.  We say that $t\in (0,T)$ is a cut time and that $W(t)$ is a cut point for the path $W[0,T]$
 up to time $T > t$ if $W[0,t) \cap W(t,T] = \eset$.
  It was first shown by Dvoretsky, Erd\H{o}s, and Kakutani \cite{DEK} 
 that if $d \geq 4$,  then all points visited by $W$ are cut points
 for $W[0,\infty)$.  If $d=1$, then cut points are the
 same as points of increase or decrease, and
 it was first shown in \cite{DEK2} that there
 are a.s.\ no such points.     If $d=2,3$,   the typical point is not a cut point, but as shown first in \cite{Bur1989}, with probability one there are cut points.  This result was improved in \cite{Law1996},
 where it was shown that with probability one the
 set of cut points in $W[0,1]$ (or in $W[0,\infty)$
 for $d=3$) has box and Hausdorff dimension  $
 \delta = \delta_d:= d - \eta$, where $\eta = 
 \eta_d = \xi + (d-2)$
  and 
 $\xi=\xi_d$ is the intersection exponent, which
 satisfies\footnote{Sometimes $ \xi/2$ is called the
 	intersection exponent; in this paper, we will always refer to $\xi$ as the
 	intersection exponent.} (see the end of this section for the notation $\asymp$)
 \[   \Prob[W[0,1-\epsilon^2] \cap W[1 +
  \epsilon^2,2] = \emptyset] \asymp \epsilon^{\xi}, \;\;\;\;\; \epsilon \downarrow 0.
  \]
  That paper did not establish the value of
  the exponent.  For $d=2$, it was shown in
  \cite{LSWTrilogy1,LSWTrilogy2} that $\xi_2=
  5/4$, proving a prediction made by Duplantier
  and Kwon \cite{DKwon}.  The value of the intersection exponent is
  not known for $d=3$, although it is known
  rigorously \cite{BurLaw1990,Law1998} that
  $1/2 < \xi_3 < 1$, and numerical simulations
  \cite{BurLawPol} suggest a value around $.58$.
  In this paper, we make a significant improvement
  on these results by establishing the existence
  and nontriviality of the Minkowski content for
  the set of cut points.

Let us describe our result precisely.
   For the remainder of this paper,
$d$ will always be either $2$ or $3$, and $\xi,
\eta,\delta$ will be as in the previous paragraph.
  We write $\Gamma$ for the set of   (continuous) curves
  $\gamma:[0,t_\gamma] \rightarrow \R^d$ and
   $\Gamma_{x,y}$ for the set of such curves with $\gamma(0) = x, \gamma(t_\gamma) = y$.
  Unless stated otherwise, we assume that the
  time duration $t_\gamma < \infty$.  If
  $\gamma \in \Gamma$, 
  let
  \begin{equation}\label{eq:cpdef}
{\cal T}={\cal T}_\gamma := \{t\in[0,t_\gamma]: \gamma[0,t)\cap \gamma(t,t_\gamma]=\emptyset\}\quad \mbox{and} \quad {\cal A} = {\cal A}_\gamma  :=\gamma(\cal{T})
\end{equation} 
denote sets of {\it cut times} and  {\it cut points} of $\gamma$, respectively. 

   Let $\e = (1,0)
\mbox{ or } (1,0,0)$ be the unit vector whose
first component equals one.   Let $\mu
 =\mu_{0,\e}$ stand for the Brownian path measure on $\Gamma_{0,\e}$
 (see Section \ref{brownsec} for definition and intuition)  corresponding to Brownian paths from 0 to $\e$. 
If $d=2$,   $\mu$ is an infinite measure while
it is finite for $d=3$.  However, although $\mu$ is infinite for $d=2$, as we will see,
 if $V$ is any
closed set disjoint from $\{0,\e\}$,
$\mu[\{\A \cap V \neq \eset\}]< \infty,$ i.e., the mass of the set of paths for which there are cut points in $V$ is finite.
The result in \cite{Law1996} implies that
$\mu$-almost everywhere (that is, except on a set of curves of $\mu$-measure zero),
${\rm dim}_h({\cal A})=\delta .$
We will give a similar but stronger result about the Minkowski content of ${\cal A}$.

Minkowski content is a natural way to define  (random) ``fractal'' measures
on random fractal sets.  
For every bounded Borel set $A \subset \R^d$, define
the  
$\delta$-dimensional Minkowski content of   $A$    by
\begin{equation}\label{eq:firstdef}
{\rm Cont}_{\delta}(A) =\lim_{r\to\infty} {\rm Cont}_{\delta}(A;r),
\end{equation}
where
\begin{equation}
{\rm Cont}_{\delta}(A;r)=e^{r(d-\delta)} \,\Vol \big\{z: {\rm dist}(z,{A})\leq e^{-r}\big\},
\end{equation}
provided that the limit exists.  
Here $\Vol$ stands for the $d$-dimensional volume  (Lebesgue measure)
in $\R^d$. 

We then focus on cut points and express ${\rm Cont}_{\delta}(\cA;r)$ in an alternative way which is easier to analyze. Let $I_s(z) =
 I_s(z;\cA)$ be the indicator function
 of $\{\dist(z,\cA) \leq e^{-s}\}$, let $J_s(z)
 = e^{\eta s} \, I_s(z)$.
and for $V$ a Borel subset of $\mathbb{R}^d$ set 
\begin{equation}\label{eq:JVdef}
   J_{s,V} = \int_V
   J_s(z) \, dz ,\;\;\;\;
\mbox{ and } \;\;\;\; J_V = \lim_{s \rightarrow \infty} J_{s,V},
\end{equation} 
    provided the limit exists. 
It is immediate by definition that  ${\rm Cont}_{\delta}(\cA;r)=J_{r,\R^d}$ and (if the limits exist) ${\rm Cont}_{\delta}(\cA)=J_{\R^d}$.

The cut-point Green's functions (one- and two-point) are defined by (provided the limits exist)
\begin{equation}\label{eq:G1def}
\greencut(z)=\lim_{s \rightarrow \infty
} \, \mu[J_s(z)] 
\end{equation}
and
\begin{equation}\label{eq:G2def}
\greencut(z,w)=\lim_{s \rightarrow \infty
} \, \mu\left[J_s(z) \, J_s(w)\right],  
\end{equation}
respectively. Here we write $\mu[X]$ for 
 the integral of $X$ with respect to $\mu$. 
 
{Finally, we define
\[   
\dist(x_1,\ldots,x_n) =  \min_{j \neq k} |x_j - x_k|
\] 
(with $|\cdot|$ the Euclidean norm) for $x_1,\dots,x_n\in\mathbb{R}^d$.}
 
We are now ready to state our first theorem on the existence of the one-point Green's function  for cut points.
 \begin{thm}  \label{theorem.dec16.2}
There exists $u >0$ such that  if $e^{-b}  =
\dist(0, z,\e)>0$, then
the limit in  \eqref{eq:G1def} 
exists. Furthermore,   if $s \geq b+1 $,
\begin{equation}\label{eq:T1c1}
\mu\left[J_s(z)\right] 
    = \greencut(z) 
     \, [1+O(e^{(b-s)u})],
\end{equation}
where for some constant $c_d>0$ depending only on the dimension $d$,
\begin{equation}\label{eq:T1c2}
 \greencut(z) = c_d |z|^{-\eta} \, |\e-z|^{-\eta}.
\end{equation}
     \end{thm}
     
Before stating our second theorem on the existence of the two-point Green's function, we need to introduce some notations. {For $n\in\mathbb{Z}^+$, let
$\dyadic_n$} denote the following set of half-open-half-closed dyadic cubes $V$ in $\R^d$  
 of the form
\begin{equation}\label{eq:dyadicV}
  V= \left(\frac{k_1}{2^n}, \frac{k_1+ 1}{2^n}\right]
   \times \cdots \times \left(\frac{k_d}{2^n}, \frac{k_d+ 1}{2^n}\right],\;k_1,\ldots,k_d\in\mathbb{Z}.
\end{equation}
{    Let $\dyadic_n^o$ be the collection of $V\in \dyadic_n$ such that $\dist(0,V)\wedge\dist(\e,V)\geq 2\,{\rm diam}(V)=\sqrt{d}2^{-n+1}$ and write 
\begin{equation}\label{eq:dyadico}
\dyadic=\cup_{n\in\mathbb{Z}^+}\dyadic_n^o.
\end{equation}} Define for $x\in\mathbb{R}^d$ and $V\subset\mathbb{R}^d$ 
\[
\dist(x,V)=\inf_{y\in V}\dist(x,y)\quad\mbox{ and }\quad{\rm diam}(V)=\sup_{x,y\in V}\dist(x,y).
\]
 Note that if  $V \in \dyadic$,  then for $z,w \in V$,
$|z-w| \leq \sqrt 3/2 \, \dist(0,\e,z,w)$.

    \begin{thm}  \label{groundhog}
     There exists $u> 0$
    such that  if $V \in \dyadic$ and     $z,w \in V$ with $|z-w| = e^{-b} > 0$,
    the limit in  \eqref{eq:G2def} exists. Furthermore,   
   if $s \geq b+1$, $0 \leq \rho \leq 1$,
\begin{equation}\label{eq:T2c1}
\mu\left[J_s(z) \, J_{s +\rho}(w)
   \right]  
    = \greencut(z,w)  \, 
      [1+O( e^{(b-s)u})].
\end{equation}    
   Moreover,  there exists
$0<c_V<C_V < \infty$ such that  
     \begin{equation}\label{eq:T2c2}
c_V \, |z-w|^{-\eta} \leq    \greencut(z,w) \leq C_V \, |z-w|^{-\eta}. 
\end{equation}\end{thm}
We could establish the limit and up-to-constant asymptotics of $\greencut(z,w)$
for all $z \neq w$ but this is all that
    we will need for the next theorem, and it makes the proof slightly easier to restrict to
    $z,w \in V \in \dyadic$.

  \begin{thm} \label{theorem.dec16}  Suppose  $d=2,3$ and $V$ is a {bounded Borel} subset of $\R^d$ such that $\partial V$ has zero $(d-\epsilon)$-Minkowski
  content for some $\epsilon > 0$.  Then, except on a set of curves of zero $\mu$-measure (abbreviated as $\mu$-a.s.\ below) the Minkowski content $ 
 \Cont_{\delta}(\A_\gamma \cap V)$ exists, is finite, and 
 equals $J_V$ (see (\ref{eq:JVdef}) for its definition). {In the meanwhile, $\mu$-a.s., $\Cont_{\delta}(\A_\gamma \cap \partial V )=0$}.  Moreover,
 \[  \mu\left[\Cont_{\delta}(\A_\gamma \cap V)\right]
  =  \mu[J_V] = \int_V \greencut(z) \, dz.\]
 In particular, $\mu$-a.s., this holds for all   $V \in \dyadic$. 
Moreover,
 if $V \in \dyadic$, 
 \[ 
   \mu\left[\Cont_{\delta}(\A_\gamma \cap V)^2\right] = \mu[J_V^2] =  \int_V \int_V \greencut(z,w) \, dz\, dw < \infty.\]
 \end{thm}
{This theorem, along with a standard procedure (see Appendix \ref{se:ApA} for more details), allows us to induce a measure out of the Minkowski content of $\cA$:
{
\begin{equation}\label{eq:Borel0}
\begin{split}
&\mbox{$\mu$-a.s.\ there }\mbox{exists a random non-atomic Borel measure $\nu$ which is }\\
&\mbox{supported on $\cA$ such that for all $V\in\cD$, } \nu(V)=\Cont_\delta (\cA\cap V);
\end{split}
\end{equation}}
moreover, for all Borel sets $U,U'$,
 \begin{equation}\label{eq:Borel}
  \mu\left[\nu(U)\right] = \int_U G^{\rm cut}(z) \, dz\ \textrm{and}\ \mu\left[\nu(U)\nu(U')\right] = \iint_{U\times U'} G^{\rm cut}(z,w) \, dz\, dw.
 \end{equation}}
 
The derivation of Theorem \ref{theorem.dec16} from
\eqref{eq:G1def}, \eqref{eq:G2def}, and the
estimates on the Green's function is essentially
the same as that followed in \cite{Law2015,LawRez} where the Minkowski content of the Schramm-Loewner evolution (SLE) path is established.   We do this in 
 Section \ref{gensec} where we give a general
 result showing that very sharp estimates on
 Green's function convergence give results about
 Minkowski content.  As in \cite{Law2015,LawRez},
 the hard work is to establish the result about the
 Green's function; in our case, this is
 Theorems \ref{theorem.dec16.2} and \ref{groundhog}.
 
 To prove these theorems, it suffices to show that there
 exists $c,u >0$ such that if $s \geq b+1$ and 
 $0 \leq \rho \leq 1,$
\begin{equation}  \label{dec16.10}
 \left|e^{   \rho\eta}
 \, \mu[I_{s+\rho}(z)] -   \mu[I_{s }(z)]\right| \leq c \, 
 \mu[I_s(z)] \, e^{(b-s)u}, 
 \end{equation}
\begin{equation}  \label{dec16.11}
 \left|e^{   \rho\eta}
 \, \mu[ I_s(z) \, I_{s+\rho}(w) ] -   \mu[I_{s }(z)
  \, I_{s }(w)]\right| \leq c \, 
 \mu[I_s(z) \, I_{s }(w)] \, e^{(b-s)u}, \mbox{ and}
 \end{equation}
 \begin{equation}  \label{dec16.12}
   \left|e^{  2 \rho\eta}
 \, \mu[ I_{s+\rho}(z) \, I_{s+\rho}(w) ] -   \mu[I_{s }(z)
  \, I_{s }(w)]\right| \leq c \, 
 \mu[I_s(z) \, I_{s }(w)] \, e^{(b-s)u}.
 \end{equation}
 
 We will now briefly outline the idea for the proofs of \eqref{dec16.10} -- \eqref{dec16.12}.  Let us write
 $\ball_k(z)$ for the closed ball of radius
 $e^k$ about $z$ with boundary
 $\partial \ball_k(z)$.  Choose large $s$ and let
 us write $\ball = \ball_{-s}(z)$.  If $\gamma$
 is a curve from $0$ to $\e$ that intersects $\ball$,
 we can decompose $\gamma$ as
 \[    \gamma = \gamma^- \oplus \omega \oplus
   \gamma^+ \]
where $\gamma^-$ is $\gamma$ stopped at the
 first visit to $\partial \ball$; $\gamma^+$ is $\gamma$ after the last exit from $\partial \ball$;
 and $\omega$ is the
 piece that connects $\gamma^-$ and $\gamma^+$.  We start by choosing
 $\gamma^-,\gamma^+$ to be independent
 Brownian motions starting at $0,\e$, respectively,
 conditioned so they reach $\partial \ball$ (this
 conditioning is not needed if $d=2$); we then condition on the event that $\gamma^+ \cap \gamma^- =\eset$; finally, given $(\gamma^-,\gamma^+)$ we consider the ratio of the measure of the set of paths $\omega$ with the property that $\gamma$ has a cut point contained in $\ball_{-s-\rho}$ to the measure of those that have
 a cut point
 in $\ball$.  If $s$ is large, then this ratio should depend only on the part of  $(\gamma^-,\gamma^+)$ near $\ball$ and this conditional distribution should be very close to a (appropriately scaled) distribution corresponding to Brownian  paths conditioned not to intersect.
 
 Similarly to get \eqref{dec16.11} and \eqref{dec16.12} we decompose paths that first
 visit $\ball_s(z)$ and then visit $\ball_s(w)$ 
 as
 \[    \gamma = \gamma^- \oplus \omega^- \oplus
    \gamma^* \oplus \omega^+ \oplus  \gamma^+
    ,\]
 where $\gamma^-$ is $\gamma$ stopped at the
 first visit to $\ball_s(z)$; $\gamma^+$ is $\gamma^\rev$ after the last exit
 from $\ball_s(w)$; $\gamma^*$ is an excursion
 between $\partial \ball_s(z)$ and $\partial \ball_s(w)$; and $\omega^-,
 \omega^+$ are the
 paths that connect.  (This representation is unique only if there is a single excursion. But in our case this will be
 true with very high probability.)  We start by
 choosing $(\gamma^-,\gamma^*,\gamma^+)$ and
 then conditioning that the three paths are
 mutually disjoint.   If $s$ is large (so that $e^{-s}$
 is small compared to $|z-w|$), we can hope that the paths near $z$ and $w$ look like independent samples from this invariant measure.  If this
 is true (with sufficiently good error estimate), then we can get our result.

 The main technical tool to establish this is the invariant measure on Brownian motions conditioned on creating a cut point and the exponential convergence to this invariant measure.

The choice of $\mu_{0,\e}$ is just for convenience. Straightforward adaptations imply results about other types of Brownian measures, such as Brownian path measure in a bounded domain, Brownian excursions, Wiener measures and Brownian bridges. 
See Section \ref{sec:bounded} for a brief discussion on how to adapt the proof to the bounded domain case.
In \cite{HLLS}, a version is needed
for Brownian excursions in two dimensions; see Section \ref{excsec} for the relevant statement.

The motivation for proving this result is more than just curiosity about Brownian paths.  Random walk paths often appear in lattice statistical models and weights are often given in terms of the number of points in particular exceptional sets.  When taking continuum limits of such models it is important to know not just the geometric object but also the limit of these occupation measures.  This enables, for example, analysis of ``near critical'' behavior as a perturbation of the continuum critical object.
The particular question in this paper arose, indeed,
in the analysis of a different model, percolation.

Thanks to the connection between planar Brownian motion and $\SLE_6$ \cite{LSWTrilogy1,LSWTrilogy2,LSWTrilogy3}, the scaling limit of macroscopic  pivotal points of critical planar percolation locally looks like 2D Brownian cut points. 
In \cite{GPS-piv}, a natural pivotal measure was constructed  as the scaling limit of the counting measure on pivotal points of the critical site percolation on the triangular lattice. 
In our paper, Theorem~\ref{theorem.dec16} for $d=2$ provides another natural measure supported on macroscopic pivotal points. 
The equivalence of these two constructions is proved in \cite{HLLS}.
Since the Minkowski content definition is more intrinsic and explicit to work with, this article provides an important input to a program of the first and fourth authors on the conformal structure of uniform random planar maps based on dynamical percolation, which is  governed by a Poisson point process with the pivotal measure as its intensity.

We start this paper by discussing the relevant facts about Brownian path
measures.  This gives path decompositions that allow one to view
a Brownian motion going near a point as two independent Brownian motions,
the ``past'' and the ``future''.  The next section reviews facts about
convergence to the stationary distribution for pairs of Brownian motions conditioned
not to intersect.  We also  adapt 
these results to pairs of Brownian motions approaching a point.
The proofs of the main estimates are done next, followed by a general discussion
of how the results on Minkowski content follow from them.  The final
subsection discusses Brownian excursions in the upper half-plane
of $\R^2$. {The {a}ppendix contains various auxiliary results on measure theoretical issues, Brownian path decomposition, and some Poisson kernel estimates.}

Finally, we briefly mention our convention on notations. Throughout this work we use $c,C,c',c_*$, etc., for positive constants that may depend on the dimension $d$ but are otherwise universal. Additional dependence will be specified at their first appearance. We write $a\lesssim b$ (resp.\ $a\gtrsim b$)  if there exist a constant $C>0$ such that $a\leq Cb$ (resp.\ $a\geq Cb$). We write $a\asymp b$ if $a\lesssim b$ and $a\gtrsim b$ {and say that $a$ and $b$ are comparable.}

{\bf Acknowledgements:}
the authors would like to thank Dapeng Zhan for comments and suggestions on a previous version of the draft, especially regarding the cut-point Green's function. We also want to thank the referee for careful reading of the paper and for numerous helpful suggestions.

\section{Brownian path measures}  \label{brownsec}
It is more convenient to view Brownian motion
as a (not necessarily probability) measure on paths.
Here we define the measures $\mu_{x,y}^D$ for domains $D$ such that $\partial D$ is a finite union of disjoint {lines and circles (for $d=2$) or planes and spheres (for $d=3$)}. 
We will consider the domains both ``inside'' and ``outside'' of the spheres. Such domains $D$ are sufficient for our purpose, but we remark that the measure can also be defined for other domains $D$, provided the domain boundary is sufficiently smooth. 
   If $D = \R^d$, we write just $\mu_{x,y}$.
   These measures are supported on $\Gamma_{x,y}$ the set of curves $\gamma (t),
   0 \leq t \leq t_\gamma$, with $\gamma(0) = x, \gamma(t_\gamma) = y$.
 They will be finite measures except for
   the case $D = \R^2$. The points $x,y$ will be distinct and may be either interior or boundary points of $D$ (or one of each).  
     The definitions are slightly different for interior and boundary points. We write  $p_t(x) =
 (2\pi t)^{-d/2} \, e^{-|x|^2/2t}$  for the density
 of Brownian motion at time $t$.  We write $\sigma$ for surface measure on spheres (area if $d=3$
 and length if $d=2$). 

If $\mu$ is a positive measure, we write $\|\mu\|$ for its total mass and $\mu^\# = \mu/\|\mu\|$ for the probability measure obtained by normalization if $\|\mu\|<\infty$.
To specify a finite (strictly) positive measure
$\mu$ it suffices to give the ordered pair $(\mu^\#,
\|\mu\|).$  If $\gamma \in  \Gamma_{x,y}$ we write $\gamma^\rev \in \Gamma_{y,x}$
for the reversal of $\gamma$:
\[        \gamma^\rev(t) = \gamma(t_\gamma - t), \;\;\; 0 \leq t \leq t_\gamma.\]
If $\mu$ is a measure on $\Gamma$ we write $\mu^\rev$ for the corresponding
measure on reversed paths,
\[    \mu^\rev(C) = \mu\{\gamma: \gamma^\rev \in C \} .\]

\begin{remark}  We will be integrating a number of measure-valued functions.  These
can be defined in a straightforward way as follows. First we define a metric on the space of curves, e.g. we say that the distance between $\gamma_1:I_1\to\C$ and $\gamma_2:I_2\to\C$ for finite intervals $I_1,I_2\subset\R$ is given by $\inf_{\psi}\sup_{t\in I_1}(|\gamma_1(t)-\gamma_2(\psi(t))|+|\psi(t)-t|)$, where the infimum is over all increasing bijections $\psi:I_1\to I_2$. Then we use the Prokhorov metric to give a measure
 on probability measures ${\mathcal P}$ on curves, which
then generates a metric on finite measures by considering a finite
measure as an element of $ {\mathcal P} \times(0,\infty)$, with the second coordinate representing the total mass of the measure. We will not give the details here but only remark that the random walk counterpart of the path measure we are going to define below is straightforward.
\end{remark}

Let $\mu_{x,y,s}$ be the measure of total mass
$p_s(y-x)$ whose corresponding probability
measure $\mu_{x,y,s}^\#$ is the Brownian bridge
measure. We can write
\[  \mu_{x,y} =\int_0^\infty \mu_{x,y,s} \, ds  \]where the integral is the limit of a Riemann sum.
  If $d=3$, this is a finite measure of
total mass
\[      G(x,y) = \int_0^\infty p_t(x,y) \, dt= \frac{1}{2\pi\, |x-y|}
  ,\]
while it is an infinite measure for $d=2$.
If $D \subset \R^d$ and $x,y\in D$, then $\mu^D_{x,y}$ is
$\mu_{x,y}$ restricted to curves that stay
in $D$.  If $d=2$, and
Brownian motion exits the domain $D$ with probability 1,
then  $\mu^D_{x,y}$ is finite
  with  the total mass given by
the Green's function  $G_D(x,y)$   normalized so that when $D$ is the unit disk,
\begin{equation}
   G_D ({0,y}) = - \frac 1{\pi} \, \log |y|.
\end{equation}
 The measures satisfy reversibility,
$\mu_{y,x}^D = [\mu_{x,y}^D]^\rev$.

We now define the interior-to-boundary measure as the rescaled limit of the measure defined above.  For $x \in  D,
	y \in \partial D$, define $\mu^D_{x,y}$ as follows,\footnote{It should be noted that, in order to keep the notation light, we do not distinguish the notation here from the case of the interior-to-interior measure $\mu^D_{x,y}$ in the previous paragraph. However, there is no confusion as soon as the location of the starting and ending points are specified.}  where $\n_y$ denotes
the inward unit normal at $y$ into $D$
   \begin{equation}\label{eq:intobddef}
       \mu_{x,y}^D := \lim_{\epsilon \downarrow 0}
   \frac{1}{2\epsilon} \, \mu_{x,y+\epsilon \n_y}^D
   \;\;\;\;
   x \in D, \;  y \in \partial D .
\end{equation}  
   Note the factor of $2$.  For example, if $D$
   is the unit disk for $d=2$ and $|y| = 1$, then
\begin{equation}\label{eq:poissonestimates}
\|\mu_{0,y-\epsilon \n_y}^{D}\| = - \frac 1{\pi} \, \log (1-\epsilon)    \sim \frac \epsilon \pi, \quad\mbox{while}\quad\|\mu_{0,y}^D\| = 1/2\pi.
\end{equation}   
     The   boundary-to-interior point measure is obtained by reversing paths:
\begin{equation}
     \mu_{y,x}^D = \left[\mu_{x,y}^D\right]^\rev
      = \lim_{\epsilon \downarrow 0}
   \frac{1}{2\epsilon} \,  \mu_{y+\epsilon \n_y,x}^D, \;\;\;\;
   x \in D, \;  y \in \partial D 
 .
\end{equation}   
  A similar ideology also allows us to define the boundary-to-boundary
measure, 
  \begin{equation}\label{eq:bdrytobdry}
         \mu_{y,x}^D = 
 \lim_{\epsilon \downarrow 0} \frac{1}{2\epsilon}
  \mu_{y+ \n_y,x}^D =  \lim_{\epsilon \downarrow 0} \frac{1}{4\epsilon^2}
  \mu_{y+ \n_y,x+ \n_x}^D, \;\;\;\; x,y \in \partial D.
\end{equation}  
  If $\partial_1,\partial_2$ are disjoint subsets of $\partial D$, we write
  \begin{equation}
  \mu_{\partial_1,\partial_2}^D =
     \int_{\partial_1} \int_{\partial_2} \, \mu_{x,y}^D\, 
      \sigma (dx,dy).
\end{equation}     
The measure $2 \,  \mu_{\partial_1,\partial_2}^D$
is  often called
 excursion measure for excursions between $\partial_1$ and $\partial_2$ in $D$.
 
 We can also define the interior-to-boundary measure in the following alternative  simple fashion.
 Let  $\mu_{x,\partial D}$ denote the 
  measure of a Brownian motion started from $x$ stopped
  when it reaches $\partial D$ (restricted to the
  event that the path leaves $D$).   Note that $\|\mu_{x,\partial D}\|$ is the probability that
  a Brownian motion starting at $x$ ever exits $D$.  
 If $x \in D, y \in \partial  D$,
 we define $\mu^D_{x,y}$ by saying that if $V \subset \partial D$, then $\mu_{x,\partial D}$ restricted to paths
 that exit $D$ at $V$, is given by \(\int_V \mu_{x,y}^D \, \sigma(dy)\).  In particular, 
  \begin{equation}\label{eq:int-bdy}
  \mu_{x,\partial D}
  =  \int_{\partial D} \mu_{x,y}^D \, \sigma(dy).
  \end{equation}

The Brownian path measures are reversible, i.e., 
$  \mu_{y,x}[C] = \mu_{x,y}\left[\{\gamma: \gamma^R \in C\} \right]. $
They are also translation invariant and satisfy the following scaling property. If $r >0$, and $f_r(z) = rz$, then we define $f_r \circ \gamma$ to be the
image reparametrized appropriately, i.e.,
\[   t_{f_r \circ \gamma} = r^2 \, t_\gamma\quad \mbox{ and }\quad
   f_r\circ \gamma(t) =  f_r (\gamma(t/r^2)) = r \,  \gamma(t/r^2).\]
If we define $f_r \circ \mu$ by 
 $f_r \circ \mu [C] = \mu\left[\{\gamma: f_r \circ \gamma \in C\}\right]$,
then 
\begin{equation}\label{eq:scalingpty}
f_r \circ \mu_{x,y} =  r^{d-2} \,\mu_{rx,ry}.
\end{equation}   
{In two dimensions, Brownain path measures also inherit the conformal invariance of the (normal) planar Brownian motion. More precisely, let $D$, $D'$ be two domains satisfying the requirement at the beginning of this section and $\phi:D\to D'$ be {a} conformal map from $D$ to $D'$. Then if $x,y\in D$, 
\begin{equation}
\phi\circ\mu^{D}_{x,y}= \mu^{D'}_{\phi(x),\phi(y)}.
\end{equation}}

The advantage of using Brownian path measures instead of usual Wiener measure or Brownian bridge measure, is that it  allows for ``decomposition'' of the path from both directions. We will give the important decompositions.  They are versions of the strong Markov property or ``last-exit decomposition'', which is the strong Markov property on the reversed path. For ease, we again assume that all of our boundary components are {lines and circles (for $d=2$) or planes and spheres (for $d=3$)}. For completeness we have included proofs of the identities in Appendix \ref{appendix}, but we also refer to \cite[Section 5]{lawler-book} for some statements when $d=2$.

\begin{itemize}
 \item Let $D$ be a ball {or half-space} in $\R^d$ (allowing $D=\R^d$) and $S$ a {line or circle (for $d=2$) or plane or sphere (for $d=3$)} in $D$. Let $D_1,D_2$ denote the components
 of $D\setminus S$. 
Then if $x \in {D_1}, y \in D_i$, $i=1,2$, 
\begin{equation}\label{eq:decomp11}
\mu^{D}_{x,y} =  \mu_{x,y}^{D_1} +\int_S \left[\mu_{x,\zeta}^{D_1} \oplus \mu_{\zeta,y}^D\right] 
\, \sigma(d\zeta)
\end{equation}
\begin{equation}\label{eq:decomp12}
\mu^{D}_{x,y} =  \mu_{x,y}^{D_1} + \int_{S} \left[\mu_{x,\zeta}^D \oplus \mu_{\zeta,y}^{D_i}\right]
\, \sigma(d\zeta) ,\mbox{ and }
\end{equation}
\begin{equation}\label{eq:decomp13}
\mu^{D}_{x,y} =  \mu_{x,y}^{D_1} +\int_S \int_S \left[\mu_{x,\zeta_1}^{D_1} \oplus
\mu_{\zeta_1,\zeta_2}^D \oplus  \mu_{\zeta_2,y}^{D_i}\right] 
\, \sigma(d\zeta_1,d\zeta_2).
\end{equation}
 In each of these formulas, the integral term represents $\mu_{x,y}$
 restricted to curves that intersect $S$ (notice that if $y\notin D_1$, then $\mu_{x,y}^{D_1}=0$).  The three integral
 terms represent decompositions based on: the first visit to $S$;
 the last visit to $S$;
  and both the first and last visits to $S$,
 respectively.
 \item  Let $S_1,S_2$ be disjoint {lines or circles (for $d=2$) or planes or spheres (for $d=3$)} in $\R^d$ and let
 $D$ be the component of $\R^d \setminus (S_1 \cup S_2)$
  whose boundary includes both $S_1$ and $S_2$.  Let
 $D_j$ be the component of $\R^d \setminus   S_{3-j}$  that contains
 $S_j$.  Then if $x \in S_1, y \in S_2 , $
\begin{eqnarray}
  \mu_{x,y} & =&
   \int_{S_1} \int_{S_2}
        \left[\mu_{x,\zeta_1} ^{D_1}\oplus \mu^{D}_{\zeta^1,\zeta^2} \oplus
           \mu_{\zeta_2,y} \right] \,  \sigma(d\zeta_1,d\zeta_2) \label{eq:decomp21}\\
           & = &  \int_{S_1} \int_{S_2}
        \left[\mu_{x,\zeta_1} \oplus \mu^{D}_{\zeta^1,\zeta^2} \oplus
           \mu_{\zeta_2,y}^{D_2} \right] \,  \sigma(d\zeta_1,d\zeta_2). \label{eq:decomp22}
     \end{eqnarray} 
The first of these expressions is obtained
by  decomposing $\gamma \in \Gamma_{x,y}$
as  $\gamma = \gamma^1 \oplus \gamma^*
 \oplus \gamma^2$ where $\gamma^*$ is the first excursion from $S_1$ to $S_2$ in $D$.
 The second uses a similar decomposition where $\gamma^*$ is the
 last such excursion. 
  
 \item  Let $S_1,S_2,D$ be as above and assume $S_3,S_4$ are {lines or circles (for $d=2$) or planes or spheres (for $d=3$)}
 in $D$ such that every path from $S_1$ to $S_2$ in $D$ must go through
 $S_3$ and then $S_4$ in that order.  Let $D'$ denote the component of
 $D \setminus S_4$ that contains $S_1$ on its boundary.
 Let $D_0$ denote the component
 of $D \setminus (S_3 \cup S_4)$ that contains  both $S_3$ and
 $S_4$ on its boundary. Then if $\zeta_1 \in S_1, \zeta_2 \in S_2 , $
 \begin{equation}\label{eq:decomp3}
  \mu^D_{\zeta_1,\zeta_2} = \int_{S_3} \int_{S_4}
  \left[\mu_{\zeta_1,\zeta_3} ^{D'}\oplus \mu^{D_0}_{\zeta^3,\zeta^4} \oplus
  \mu_{\zeta_4,\zeta_2}^D \right] \,  \sigma(d\zeta_3,d\zeta_4).
 \end{equation}
 This is obtained by writing an excursion $\gamma$ in $D$ from $S_1$ to $S_2$
 as  $\gamma = \gamma^1 \oplus \gamma^*
 \oplus \gamma^2$ where $\gamma^*$ is the first excursion between $S^3$ and $S^4$
 in $D_0.$
  \item  Let $S_1,S_2,S_3,S_4,D$ be as above, let $D_3$ be the domain bounded 
  by $S_1$ and $S_3$ and let $D_4$ be the domain bounded by
  $S _4$ and $S_2$.  Then if $\zeta_1 \in S_1, \zeta_2 \in S_2 , $
\begin{equation}  \label{jan22.1}
 \mu^D_{\zeta_1,\zeta_2} = \int_{S_3} \int_{S_4}
   \left[\mu_{\zeta_1,\zeta_3} ^{D_3}\oplus \mu^{D }_{\zeta^3,\zeta^4} \oplus
           \mu_{\zeta_4,\zeta_2}^{D_4} \right] \,  \sigma(d\zeta_3,d\zeta_4).
           \end{equation}
  Here we write  $\gamma = \gamma^1 \oplus \gamma^*
 \oplus \gamma^2$ uniquely where $\gamma^1$ is an excursion from $S_1$
 to $S_3$ in $D_3$ and $\gamma^2$ is an excursion from $S^4$ to $S^2$
 in $D_4$. 
     \end{itemize}

We will need an invariance under inversion that is true for $d=2,3$.  Let 
$B_t$ be a Brownian motion with $ B_0 \neq 0$, and define    
\[          Y_t =  \frac{B_{\sigma(t)}}{|B_{\sigma(t)}|^2}, \;\;\;
     \mbox{  where  } \;\;\;    \int_0^{\sigma(t)} \, \frac{ds}{|B_{s}|^2}
      = t . \]
      
As in the introductory section we write
 $\ball_k(\cdot)$ for balls of the exponential scaling and omit the center when it is clear from context, e.g., from here until the end of this section, unless otherwise indicated all balls are centered at the origin. We let $T_k=T_k(B)$ denote the time at which $B$ hits $\partial\ball_k$. 
      
 \begin{proposition}   \label{invertprop}
 Suppose $d =2,3$ and $B_t$ is a standard Brownian motion
 with $0 < |B_0| < e^k.$  Then the distribution of
 $$              Y_s  , \;\;\;\; 0 \leq s \leq \sigma^{-1}(T_k(B))=T_{-k}(Y), $$ 
 is the same as that of a Brownian motion starting at $B_0/|B_0|^2$, stopped
at time $T_{-k}$, conditioned that $T_{-k} < \infty$. 
\end{proposition}

Note that when $d=3$, the law of $Y$ is given by $\mu_{z,0}$ after normalization, where $z=B_0$. Also note that if $d=2$ then $\Prob[T_{-k} < \infty] = 1$ so we do not need
to make the conditioned statement.  It is necessary for $d=3$.

\begin{proof}  This uses the representation of a $d$-dimensional Brownian motion
as the product of a ${\rm Bes}(d)$ process for the radial part  and an independent
spherical part.  The corresponding fact about Bessel processes for $d=2,3$ is well
known and follows from an standard It\^o's formula  calculation
 (see for instance  \cite[Prop 1.1 \small{(ii)}]
 {Pit1975} or  \cite[Theorem 3.6]{Wil1974}). 
 \end{proof}

If $d=2$, the path measure $\mu_{x,y}$ is infinite.  However, if we restrict
to a particular set of paths, we often get a finite measure.  The following
two lemmas are examples of this that will be important to us.

\begin{lemma} \label{dec8.lemma1}
 If $d=2$ and $x,y$ are distinct points,  let $C = C_{x,y}
$ denote the set of paths $\gamma(t), 0 \leq t \leq t_\gamma$, such that $\gamma$  does not 
make two closed loops contained in the annulus
$A$ separating its two boundary circles, where $A:= A_{x,y} =  \{z:  |x-y| \leq  |z-(x+y)/2| \leq 2 \, |x-y|\}$.
  To be more precise, $C $ is the
set of curves from $x$ to $y$ such that  there do not exist 
$0 < t_1 < t_2 < t_3 < t_4 < t_\gamma$ such that $\gamma[t_1,t_2] \subset A,
\gamma[t_3,t_4] \subset A$ and such that $(x+y)/2$ 
is not in the unbounded component
of either $\R^2 \setminus \gamma[t_1,t_2]$ or $\R^2 \setminus
\gamma[t_3,t_4]$.  Then $ c':=\mu_{x,y}[C] < \infty$ and does not depend on $x,y$.

\end{lemma}

The geometric significance of the event $C $ is as follows. 
Suppose $\gamma^1,\gamma^2$ are curves starting outside $\ball_2$ 
and stopped at $\partial\ball_0$.  Let $x,y \in \partial\ball_0$
be the terminal points. 
   If   $\wt \gamma \in \Gamma_{x,y} \setminus C $,
there exist  $s < t$ such that
$\wt \gamma(s) \in \gamma^2, \wt\gamma(t) \in \gamma^1$.
In particular,   the 
concatenation  $\gamma = \gamma^1 \oplus \wt \gamma \oplus [\gamma^2]^R$ 
has no cut points in $\ball_0$. 
\begin{proof} 
	The fact that $\mu_{x,y}[C] $ is independent of $x,y$ follows
from translation invariance and scaling, so we may assume that $x=(-1,0), y = (1,0)$.
For each path we consider the excursions between $\partial\ball_2$ and $ \partial 
\ball_1$, that is,   let $\tau_0 = \sigma_1 = 0$ and 
\[   \tau_1 = \min\{t: |\gamma(t)| = e^2\}, \]
and for $k \geq 2$,
\[     \sigma_{k} = \min\{t \geq \tau_{k-1}: |\gamma(t)| = e\},
\;\;\;\; \tau_k = \min\{t \geq \sigma_k: |\gamma(t)| =  e^2\}. \]
Let $U_n$ be the set of curves from $x$ to $y$ such that $\tau_n < \infty, 
\tau_{n+1} = \infty$.  Let $q > 0$ be the probability that a Brownian motion
starting from a point on $ \partial \ball_1$ makes two loops in $A$ before reaching
$\partial \ball_2$.  We now claim the following, which immediately implies the lemma 
\eqb
\begin{split}
	\mu_{x,y} [U_0 \cap C] &\leq \mu_{x,y}[U_0]
	= G_{\ball_2}(x,y),\\
	\mu_{x,y} [U_n \cap C] &\leq (1-q)^{n-1} \max_{|z| = e}
	G_{\ball_2}(z,y) \leq c \, (1-q)^{n-1},\,\,n>0.
\end{split}
\label{eq:paths-exponential-decay}
\eqe
 To prove the claim, we split the path into its excursions between $\partial\ball_1$ and $\partial\ball_2$. More precisely, in, say, the case of $U_1$ we consider the following segments: (i) the path until hitting $\partial\ball_1$, (ii) then the path until hitting $\partial\ball_2$, (iii) then the path until hitting $\partial\ball_1$, and (iv) finally, the path until hitting $y$. Recall from the definition of $\mu_{x,\partial D}$ earlier in this section that measures in (i)-(iii) are all probability measures if we do not impose the restriction that the path in (ii) does not make two loops around $A$. Imposing this restriction reduces the measure by a factor of $1-q$. We obtain \eqref{eq:paths-exponential-decay} by using that the measure of path segments of the form (vi) is bounded by $\max_{|z| = e}
  	G_{\ball_2}(z,y)$.
\end{proof}
\begin{lemma} \label{feb1.1}
	Suppose $d=2$ and $V$ is a compact subset of $\R^2$ disjoint from $\{0,\e\}$.
	Then 
	\[   \mu\left[\{\gamma: \A_{\gamma} \cap V \neq \eset \}\right] < \infty . \]
\end{lemma}
\begin{proof}  Let $\ball^0, \ball^\e$ denote closed balls about $0$ and $\e$
	of radius less than $1/2$ that do not intersect $V$.  Consider $C'$,
	the set of curves $\gamma$ such that there do not
	exist  $0 < t_1 < t_2 < t_3 < t_4 $ such that
	\begin{itemize}
		\item  $\gamma[t_1,t_2] \subset \ball^\e$ and disconnects $\e$ from infinity, and
		\item  $\gamma[t_3,t_4] \subset \ball^0$ and disconnects $0$ from infinity.
	\end{itemize}
	We claim that if $\gamma \in  \Gamma_{0,\e} \setminus
	C'$, then $\A_\gamma \cap V= \eset$. To see this, 
	let $t_-$ be the first time that $\gamma$ hits $\gamma[t_3,t_4]$ and
	$t_+$ the last time that $\gamma$ hits $\gamma[t_1,t_2]$.  Then
	$\gamma(t_-) = \gamma(t')$ for some $t_3 \leq t' \leq t_4$
	and $\gamma(t_+) = \gamma(t'')$ for some $t_1 \leq t'' \leq t_2$. 
	In particular, there are no  cut times in $(t_-,t_3)$ and there are no
	cut times in $(t_2,t_+)$, which implies that there are no cut times
	in $(t_-,t_+)$.  But $(\gamma[0,t_-] \cup \gamma[t_+, t_\gamma])
	\cap V = \eset.$ Considering excursions between $\partial\ball_1$ and $\partial\ball_2$ again and the partitioning of $\Gamma_{0,\e}$ into $U_k$ as in the proof of  Lemma \ref{dec8.lemma1} and arguing similarly,
	we see that for the event $C'$ we are considering here it still follows that  $\mu[U_k\cap C']\leq c'(1-q')^{k-1}$ for some $c'>0$, $0<q'<1$. Hence, we see that
	$\mu[C'] < \infty$.
\end{proof}
The following lemma will be used in the proof of Proposition \ref{prop101} and Theorem \ref{groundhog}.
\begin{lemma}   \label{dec8.lemma2}
	There exists $c < \infty$ such that the following holds.
	Suppose
	$\wt \gamma^1,\wt \gamma^2$ are two curves connecting the
	boundary components of the annulus $\{1 < |z| < e^k\}$, ending at $x,y\in\ball_0$, respectively. Consider the set $C''{\subset}\Gamma_{x,y}$ of curves $\omega$ such that the following two facts hold: 
	\begin{itemize}
		\item    $\omega  \cap   \partial \ball_k \neq \eset $.
		\item  There do not exist $0 < t^- < t^+ < t_\omega$ such that
		$\omega(t^-) \in \wt \gamma^2, \omega(t^+) \in \wt \gamma^1$. 
	\end{itemize}
	Then $\mu_{x,y}[C''] \leq  c e^{-k(d-1)/{3}}$ where $c$ does not depend on $x,y$.
\end{lemma}
\begin{proof}  
When $d=2$, the proof is similar to that of the proof of  
	Lemma \ref{dec8.lemma1}. We consider $U'_n$, the set of curves from $x$ to $y$ with exactly $n$ excursions between $\partial\ball_1$ and $\partial \ball_k$. Since $U'_0\cap C''=\emptyset$ by definition, we only need to consider $\mu[U'_n \cap C'']$ for $n\geq 1$. When $n=1$, we can decompose $\omega$ as $\omega_1\oplus\omega_2\oplus\omega_3\oplus\omega_4$, where $\omega_1$ is the part of $\omega$ between $x$ and the first visit to $ \partial \ball_1$, $\omega_2$ the part between first visit to $ \partial \ball_1$ and first visit to $\partial  \ball_k$, $\omega_3$ the part between the first visit to $\partial  \ball_k$ and the subsequent visit to $\partial \ball_1$, and $\omega_4$ the rest of $\omega$. We can regard $\omega_1$, $\omega_2$, and $\omega_3$ as independent Brownian motions, the latter two conditioning on their respective starting points. By Beurling estimates, we see that both the probability for $\omega_2$ not to hit $\wt \gamma^2$ and for $\omega_3$ not to hit $\wt \gamma^1$ are $O(e^{-k/2})$. Hence, 
$$
\mu[U'_1\cap C'']\leq ce^{-{k/2}} \max_{|z|=e} G_{\ball_k}(z,y)\leq c' e^{-{k/2}} k,	
$$
Arguing similarly as in the proof of Lemma \ref{dec8.lemma1}, we also have for all $n\geq 1$,
$$
\mu[U'_n\cap C'']\leq c e^{-n{k/2}}  \max_{|z|=e} G_{\ball_k}(z,y)\leq c' e^{-n{k/2}}k.
$$
The claim then follows by summing up the inequalities above for $n\geq1$.

When $d=3$ we can use a much easier estimate:  if $x,y \in \partial \ball_0$, then the $\mu_{x,y}$ measure of
curves that touch $ \, \partial \ball_k$ is $O(e^{-k})$, regardless of the exact position of $x$ and $y$.
\end{proof}

\section{Invariant measure for non-intersecting Brownian paths}\label{sec:3}
In \cite{Law1995} and \cite{Law1998},  a quasi-invariant probability
measure was introduced for pairs of Brownian motions in $2$ or $3$ dimensions
conditioned not to intersect.  To be more precise,
let $\state^*$ denote the collection of
ordered pairs $(\gamma^1,\gamma^2)$,
where $\gamma^j:[0,t_{ j}] \rightarrow \R^d$ with
$\gamma^j(0) = 0, |\gamma^j(t_{j})| = 1, |\gamma(t)| < 1$ for $t <
t_{ j}$, and $\gamma^1(0,t_1] \cap \gamma^2(0,t_2]
 = \emptyset$. 

The quasi-invariance can be stated as following.  Suppose $(\gamma^1,\gamma^2)$
have distribution $\invprob^*$ and independent Brownian motions are started at
$\gamma^1(t_1), \gamma^2(t_2),$ and stopped when they reach $\partial \ball_a$ for some $a>1$. Let
$(\wt \gamma^1,\wt \gamma^2)$ be the paths obtained by concatenating
the original paths with  these
Brownian motions.
   Then  the probability
that $\wt \gamma^1(0,\wt t_1] \cap \wt \gamma^2(0,\wt t_2]
 = \eset$ is $e^{-a\xi}$, and conditioned on this event, 
 if we  use Brownian scaling 
so that the concatenated paths  are in $\state^*$,
the conditional distribution on
 the curves is $\invprob^*$. 
 
Similarly, let $\state^*_k$ denote the collection of such ordered pairs where
the paths start at $\partial \ball_{-k}$.  If $\bar \gamma \in \state^*$,
then we obtain $\bar \gamma_k \in \state_k^*$ by considering the
paths starting at their first visit to $\partial \ball_{-k}$. We write  $\invprob^*_k$ for the probability measure on $\state^*_k$ induced by $\invprob^*$.

Let us review the important facts for our paper.
Suppose $V_0,V_1,V_2 $ are pairwise disjoint subsets of $\ball_{-n}$ such that $V_0$, $V_0\cup V_1$, and $V_0\cup V_2$ are closed sets
and $x_j \in V_j$, $i=1,2$.  Let $W^j$ be independent Brownian motions starting at
$x_j$ and let $T_{-k}^j$ be the first time that the paths reach $\partial \ball_{-k}$
where $k \leq n-1$. We refer to such quintuples $(V_0,V_1,V_2,x_1,x_2)$ colloquially as initial configurations. Let  $\Lambda^j_k = V_j \cup W^j[0,T_{-k}^j]$ and 
let $A_k$ be the event  
$$A_k = \left\{(V_0\cup\Lambda^1_k)\cap \Lambda^2_k =(\Lambda^1_k)\cap (V_0\cup\Lambda^2_k) = \eset \right\}.$$
To make sure the following results make sense, we only consider initial configurations such that $A_0$ has positive probability.

Let $\wt \invprob_k = \wt \invprob_k(V_0,V_1,V_2,x_1,x_2)$ be the
probability measure on ${\cal X}_k^*$ induced by the measure of the paths  $W^1[T_{-k}^1,T_0^1] , W^2[T_{-k}^2,T_0^2]$ conditioning on the event $A_0$.  Let
\begin{equation}\label{eq:Deltadef}
\Delta_k =   e^{k} \, \left( \dist[W^1(T_{-k}^1), \Lambda^{2}_k ] \wedge
\dist[W^2(T_{-k}^2), \Lambda^{1}_k ]\right).
\end{equation}

We recall some facts about non-intersecting Brownian motions.  Note that the constants $c_1,c_2,c,u$ below are all universal, depending on dimension only. 
   
 The key steps in \cite{Law1996} to establish the
 Hausdorff dimension of cut points were the following  estimates.  The first is almost ``obvious'': if two paths avoid each other then there is a good chance they are far apart. Note that here our setting is more general than that of the references given below (in \cite{Law1996} no initial configuration was considered, in \cite{Law1995,LawVer} the authors only consider the case $V_0=\eset$ and $x,y$ on $\partial \ball_{-k}$), but the proofs carry through in exactly the same manner. In Sections \ref{sec:bounded} and \ref{excsec} there will indeed be scenarios where taking non-empty $V_0$'s is necessary.

\begin{itemize}
 \item {\bf Separation Lemma}
  (\cite[Lemma 3.4]{Law1995}, \cite[Lemma 3.2]{LawVer}):
  There exists $c > 0$ such that
 \begin{equation}
  \Prob\big[\Delta_{k} \geq 1/10 \big| A_k\big] \geq c.\label{eq:sep1}
\end{equation}    
   \end{itemize}

 \begin{itemize}  
 \item {\bf Separation-at-beginning Lemma}: 
  There exists $c > 0$ such
 that if $x_1,x_2 \in \partial \ball_{-n}$ and 
  $\Delta_n \geq 1/10$, then 
 \begin{equation}\label{eq:sep2}
  P\big[I_k\big| A_k\big]\geq c,
\end{equation}   
where $$I_k:=\Big\{W^j[0,T_k^j] \cap \ball_{-n}  \subset  \{z: |x_j - z| \leq e^{-n}/100\} , \;\; j=1,2\Big\}.$$
{For $d=3$, see \cite[Proposition 3.10]{LawVer}. For $d=2$, one can run a similar argument, where one replaces \cite[Corollary 3.9]{LawVer} by \cite[Lemma 3.8]{Law1995} as the main ingredient.}
\item{\bf Separation-at-both-ends Lemma} (\cite[Lemma 3.7]{Law1996}): Occasionally we need to use the two separation lemmas at the same time. Assuming the same as the separation-at-beginning lemma, there exists $c>0$ such that
 \begin{equation}\label{eq:sep3}
  P\big[\{\Delta_k\geq 1/10\}\cap I_k\big| A_k\big]\geq c.
\end{equation}

\end{itemize}
By the Markov property of Brownian motion, in the case that $V_1$, $V_2$ are single points on $\partial \ball_{{-n}}$ sufficiently separated one can get submultiplicativity, which implies the existence of the intersection exponent $\xi$ and the estimate  (see \cite{LSWTrilogy1})
\begin{equation}\label{eq:rough}
P(A_0) = e^{-(\xi+o(1)) n}.
\end{equation}
For general $(V_0,V_1, V_2, x_1, x_2)$, the estimate \eqref{eq:rough}
together with Markov property implies that there is a uniform constant $c_2 \in (0,\infty)$ such that $P[A_0] \leq  c_2 P[A_{n-1}]e^{-\xi n}$; and the estimate \eqref{eq:rough} together with the two separation lemmas {(see e.g.\ \cite[Proposition 3.11]{Law1995} to see how this is achieved)} implies that there is a uniform constant $c_1 \in (0,\infty)$ such that $P[A_0] \geq  c_1 P[A_{n-1}]e^{-\xi n}$. 
Thus we have obtained the following up-to-constant estimate:
\begin{itemize}
 \item There exist $0 < c_1 < c_2 < \infty$ such that
\begin{equation}\label{eq:niuptoc}
            c_1 \, \Prob[A_{n-1}] \, e^{-n\xi}
   \leq \Prob[A_0] \leq c_2 \, \Prob[A_{n-1}] \, e^{-n \xi}.
\end{equation}
\end{itemize}
{Equivalently, we can rephrase this estimate as follows:
\begin{itemize}
 \item For any initial configuration $(V_0,V_1, V_2, x_1, x_2)$, there exists 
  $0 < c_1 < c_2 < \infty$ depending on the initial configuration such that
\begin{equation}\label{eq:uptocstproba}
            c_1  \, e^{-n\xi}
   \leq \Prob[A_0] \leq c_2\, e^{-n \xi}.
\end{equation}
\end{itemize}}
 The final fact deals with convergence to the invariant 
 distribution.
 
 \begin{itemize}
 \item   Recall the definition of $\invprob_{k}^*$ and  $\wt \invprob_k   =\wt \invprob_k(V_0,V_1,V_2,x_1,x_2) $ at the beginning of this section. There exist universal constants $c,u > 0$  such that
 \begin{equation}\label{eq:totalvariation}
d_{\op{TV}}\left(\wt \invprob_{n/2}
,\invprob_{n/2}^*\right)<c \, e^{-n u}.
\end{equation}  
Here and throughout this work we use $d_{\op{TV}}(\cdot,\cdot)$ to denote the total variation distance between two probability measures.
\end{itemize}

In \cite{Law1995} and \cite{Law1998} the
measure $\invprob^*$ was constructed and it was shown that the total
variation distance goes  to zero uniformly, but the estimate
was not exponential.  For the exponential rate see \cite{LSWTrilogy3} for $d=2$
and \cite{LawVer} for $d=3$.  These later papers actually treat a slightly more
general situation, but a particular case of them gives the marginal distribution
on one path, and the distribution of the second path given the first path,
is determined as an $h$-process. These arguments could be simplified
somewhat in our case.  Let us summarize the main idea.

Let us
write $\partial_k = \partial  \ball_{-k}$. Consider Brownian motions $W^1,W^2$ starting on $ \partial_n$.  We
will consider the conditional distribution on $W^1[0,T_0^1],
W^2[0,T_0^2]$ given $A_0$. 
We consider this as a probability measure on ordered pairs of curves which depends on an initial configuration $(V_0,V_1, V_2, x_1, x_2)$. Let $\bar\gamma=(\gamma_1, \gamma_2)$ and $\bar\gamma'=(\gamma'_1, \gamma'_2)$ be two pairs of random curves, whose laws are such measures with possibly different initial configurations.
Let
$\bar \gamma_k =(\gamma^1_k, \gamma^2_k)$ and $\bar \gamma'_k =(\gamma^{\prime 1}_k, \gamma^{\prime2}_k)$ be the curves stopped 
at time $T^j_{-k}$.  We write $\bar \gamma_k =_j \bar \gamma_k '$ if 
the pairs of paths agree from their first visit to $\partial _{j+k}  $
to their first visit to $\partial_{k}$. Both the existence of the invariant measure and the estimate on the rate of convergence come from the following fact.
\begin{itemize}
\item 
 If $\bar \gamma, \bar \gamma'$ have
different initial configurations in $\ball_{-n}$, we can couple
$\bar \gamma_0,\bar \gamma'_0$ such that they each have the distribution of the paths given $A_0$ and such that, except
on an event of probability $O(e^{-un})$,
\[      \bar \gamma_0 =_{n/2} \bar \gamma'_0 .\]
\end{itemize}

Let us explain a little bit about this coupling.  The two
distributions we are coupling are those of the pairs
of paths given $A_0$.  We view
$\bar \gamma_{n}, \bar \gamma_{n-1},\bar \gamma_{n-2},
\ldots, \bar \gamma_0$ as a non-homogeneous Markov chain
whose states are pairs of paths.  When we write
$\bar \gamma_k = _j \bar \gamma_k'$ we mean that the
paths agree from their first visit to $\partial_{k+j}$
to the first visit to $\partial_k$. The time durations of
the paths up to the first visit to $\partial_{k+j}$ may
be different. 

Unlike some coupling rules, this coupling allows for occasional
decoupling of paths.  Suppose that $\bar \gamma_k =_j
\bar \gamma_k'$, that is, the paths agree from $\partial_{k+j}$
to $\partial_k$.  Then one can show that, the conditional
distributions of the remainders of the paths have  total variation distance bounded above by $O(e^{-\alpha j})$.  This takes a
little work, but roughly this is because paths that are not
going to intersect do not want to return to $\partial_{k+j}$.  Given
this, we can couple $\bar \gamma_{k-1}$ and $\bar \gamma_{k-1}'$
such that, except on an event of probability $O(e^{-\alpha j})$,
we have $\gamma_{k-1} =_{j+1} \gamma_{k-1}'$.  If they are
decoupled, we say $\gamma_{k-1} =_0 \gamma_{k-1}'$.  If the
paths are decoupled, one can use the separation lemma followed
by the separation at beginning lemma to find a way to couple
the paths in the next step with positive probability.   Once one has this, one compares to a simple Markov chain on the integers that moves from $j$ to $0$ with probability $O(e^{-\alpha j})$
and otherwise from $j$ to $j+1$. It is not hard to see that this Markov chain has probability at least $1-O\big(e^{-un}\big)$ to be at a location to the right of $n/2$ at time $n$ for some $u>0$, regardless of where it started. {For a proof of this, see Proposition 2.32 of \cite{Law2020}.}

Although not needed in this work, it is worth mentioning that a variant of the above coupling scheme\footnote{In this variant, we construct the quasi-invariant  measure  $\invprob^*$ on $\state^*$ by finding a consistent family of positive measures  $\invprob^*_k$ on $\state^*_k$ rather than working with a probability measure out of conditioning only.} gives the following precise asymptotics of the non-intersection probabilities. There exist a universal $u>0$ and
$$c = c(V_0,V_1,V_2,x_1,x_2)
 \asymp \Prob[A_{n-1}]$$ such that 
\begin{equation}
\Prob[A_0] = c  \, e^{-\xi n} \, [1 + O(e^{-un})],
\end{equation}
where the error term is uniform. 
 
 We will need a variant of this proposition where the Brownian
 motions tend to a point in $\R^d$ rather than to infinity.  The results follow
 almost immediately given invariance of (time changes of)
 Brownian paths in $d=2,3$ under
  inversion, see  Proposition \ref{invertprop}.

We consider the set of paths ``started at infinity stopped when
  they reach the sphere of radius $e^k$ about the origin'' and denote it by $\newstate_k$.  Let $\newstatetwo_k$ denote the set of ordered pairs of paths $\bar \gamma = (\gamma^1,\gamma^2)\in \newstate_k\times\newstate_k$. The probability measure $\invprob^*$ induces via inversion a measure $q_k$ on $\newstatetwo_k$.  We normalize so that $\|q_0\| = 1$; in this case, we also denote it by $\invprob$. For other $k$, we normalize so that $\|q_k\| =e^{k(\eta+d-2)}$. 
  If $j < k$, to get $q_j$ from $q_k$ we do as before:
\begin{itemize}
\item  Choose $(\gamma^1,\gamma^2)$ from $q_k$, and let $x,y$ be
    the endpoints. 
\item  Start independent Brownian motions at $x,y$;
stop them when they reach $\ball_j$; concatenate these with the original
paths; and kill the process if one of the paths does not reach $\ball_j$ or
if the concatenated paths intersect.
\item  Then the measure restricted to pairs that have not been killed
is $q_j$. 
\end{itemize}
 
    Let $\doublestate$ be the set of doubly infinite curves $\gamma:(-\infty,\infty)
   \rightarrow \R^d$ with $\gamma(t) \rightarrow \infty$
   as $t \rightarrow \pm \infty$.
   We consider two such curves the same if they are time translates of each other. Let $\doublestate_k$ be the set of such curves
    that intersect $\ball_k$; in this case there is a first and last intersection,
    and hence we can write $\gamma$ uniquely as 
    \[     \gamma = \gamma^1 \oplus \omega \oplus [\gamma^2]^R, \]
where $(\gamma^1,\gamma^2)\in\newstatetwo_k $. 
     Let $\doublestate^k \subset 
    \doublestate_k$ be the set of such curves that have a cut point in
    $\ball_k$.  We define the measure $\pi_k$ on $\doublestate^k$ as follows:
    \begin{itemize}
    \item     Choose $(\gamma^1,\gamma^2)$ from $q_k$, and let $x,y$ be
    the endpoints.
    \item  Choose $\omega$ from $\mu_{x,y}$ and then restrict the measure
    to curves 
    $\gamma = \gamma^1 \oplus \omega \oplus [\gamma^2]^R \in \doublestate^k$.
   \end{itemize}
   Note that if $j \leq k$, we can also get $\pi_j$ by
   changing the second bullet to
   \begin{itemize}
   \item     Choose $\omega$ from $\mu_{x,y}$ and then restrict the measure
    to curves 
    $\gamma = \gamma^1 \oplus \omega \oplus [\gamma^2]^R \in \doublestate^j$.
    \end{itemize}
   We get the following properties.
   \begin{itemize}
   \item  There exists $c_*\in(0,\infty)$ such that 
\begin{equation}\label{eq:cstardef}
   \|\pi_k \| = c_* \, e^{\eta k}.
\end{equation}   Note that here the factor of $e^{+k(d-2)}$ cancels for $d=3$  because of the scaling rule
   for $\mu_{x,y}$. To see $c_*<\infty$, one simply applies Lemma \ref{dec8.lemma2}. The fact that $c_*>0$ seems trivial given the existence of cut points for Brownian motion, but it does require some justification.  As some ingredients will not be available until Section \ref{sec:4}, we will postpone  its proof to Section \ref{sec:cstar}. Note that nothing in between requires this fact. 
   \end{itemize}

   We now consider paths of finite length.  Let $\wh \state_k$ denote the set of
 ordered disjoint pairs of curves  $\bar \gamma
 = (\gamma^1,\gamma^2)$ starting in $\ball_k^c$ and ending at their first
 visit to $\partial \ball_0$; let $\state_k$ be the set of such
 ordered pairs of  curves that start on $\partial \ball_k$.  For each $\bar \gamma \in \wh \state_k$,
 there is a unique $\bar \gamma^{(k)} \in \state_k$ obtained by starting the
 curves at their first visits to $\partial \ball_k$.  If $\bar \gamma_1,
 \bar \gamma_2 \in \wh \state_k$, we write $\bar \gamma_1 =_k
 \bar \gamma_2$ if $\bar \gamma_1^{(k)} = \bar \gamma_2^{(k)}$, that is, if
 the curves agree starting at the first visits to $\partial \ball_k$.
 
   There is a natural bijection between
 $\state_k$ and $\state_k^*$ obtained from inversion as in
 Proposition \ref{invertprop}  
 (being careful about the time change). 
 Let 
  $\invprob  _k$ 
 denote the probability measure on $\state_k$ induced from $\invprob_k^*$.
 With slight abuse of notation, for any $j > 0$ and $x \in \R^d$,
 we will also write  $\invprob_k$ and
 $\state_k$   for the
 corresponding measures
 and sets of curves from $\partial \ball_{k-j}(x)$ to $\partial \ball_{-j}(x)$
 obtained by Brownian scaling and translation. 
 
 If $\bar \gamma = (\gamma^1,\gamma^2)  \in \wh \state_k\mbox{ or }\newstatetwo_0$ with terminal points $ x_1,x_2
 \in \partial \ball_0$,  and $a > 0$, let $\mu^{\bar \gamma,a}$
 denote $\mu_{x_1,x_2}$ restricted to those curves $\omega$  such that there exists a cut
 point of $\gamma^1 \oplus \omega \oplus [\gamma^2]^R$ contained in $\ball_{-a}$.
 Let 
 \begin{equation}\label{eq:Phidef}
 \Phi_a(\bar \gamma) := \big\|\mu^{\bar \gamma,a}\big\|\quad\mbox{ and }\quad  \wh\Phi_{a,k}(\bar \gamma) := \big\|\mu^{\bar \gamma,a}[1_{\omega\in \ball_k}]\big\|.
 \end{equation}
It is easy to see that  for $a>0$, $\invprob[\Phi_{a}]$ is equal to  {$c_*e^{-a\eta}$} from \eqref{eq:cstardef}.
We write
 $\Phi_a^{(k)}(\bar \gamma) = \Phi_a(\bar \gamma^{(k)})$ where $\bar \gamma^{(k)}$
 is the truncated version as above.  Note that $\Phi_a^{(k)} \geq \Phi_a$. Observe that if $\omega\subset\ball_{k/2}$, then a cut point of $\gamma^{1,k/2}\oplus\omega \oplus [\gamma^{2,k/2}]^R$ is also a cut point of $\gamma^1 \oplus \omega \oplus [\gamma^2]^R$. See Figure \ref{fig:1}. Applying  Lemma \ref{dec8.lemma2} to $\omega$, we see that there exists   
  $c > 0$ such that if $k \geq 1$, then
 \begin{equation}\label{eq:localcp}
    \Phi_a(\bar \gamma )
  \leq \Phi_a^{(k/2)}(\bar \gamma) \leq \wh\Phi_{a,k/2}(\bar \gamma )
   +   c \, e^{-k(d-1)/{6}}
\leq  \Phi_a(\bar \gamma )
   +   c \, e^{-k(d-1)/{6}}.
 \end{equation} 
See Figure \ref{fig:1} for an illustration. 
As a consequence of \eqref{eq:localcp},
\begin{equation}\label{eq:localcpE}
       \invprob\left[\Phi_a\right] \leq \invprob\left[\Phi_a^{(k/2)}\right]
   \leq  \invprob\left[\Phi_a\right] +   c \, e^{-k(d-1)/{6}}.
\end{equation} 
\begin{figure}[h]\label{fig:1}
\centering
\includegraphics[scale=0.5]{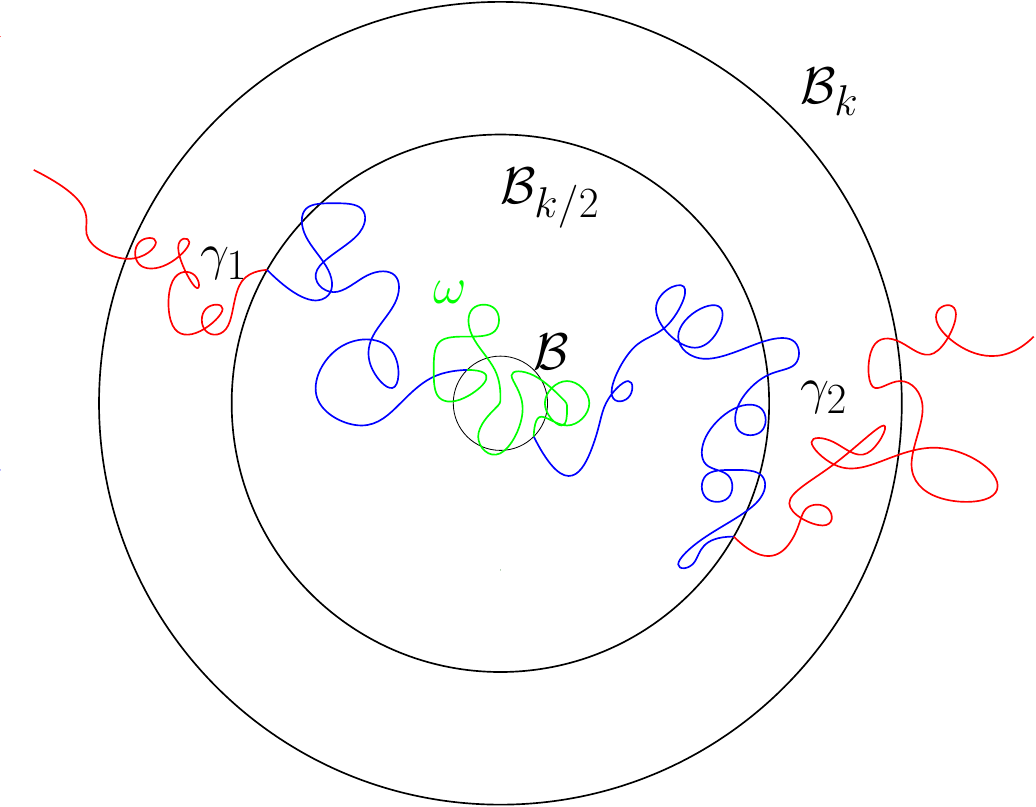}\caption{\small{ We further decompose $\gamma^1$ and $[\gamma^2]^R$ at first entry of ${\cal B}_{k/2}$. Under the conditioning and restriction, excluding an exceptional set with fast-decaying measure (with the help of Lemma~\ref{dec8.lemma2}), $\omega$ stays in ${\cal B}_{k/2}$. This implies that in this case, if there is a cut point of ``blue-green-blue'' part of the path on the segment $\omega$, it is a global cut point of $\gamma^1\oplus\omega\oplus[\gamma^2]^R$ as well. This implies the second inequality of (\ref{eq:localcp}).}}
\end{figure}   

We now state and prove the main claims of this section.
     \begin{proposition}  \label{prop101}
For $k\geq 1$ and $Q$  a probability measure on $\wh \state_k$, let $Q_k$ denote the measure induced on $\state_k$ by $Q$.
   Then there exists $c > 0$ such that if $a \geq 1/10$
(see (\ref{eq:cstardef}) for the definition of $c_*$), 
\begin{equation}\label{eq101}
   \left|Q[\Phi_a]-c_* \,e^{-a\eta} \right| \leq  
 c  \left[ e^{-k(d-1)/{6}} + d_{\op{TV} }\left(Q_{k/2} , \invprob_{k/2}\right)\right].
\end{equation}
 \end{proposition}
 The condition $a \geq 1/10$ is only  needed for $d =3$
  to guarantee that $\Phi_a$
 is uniformly bounded. 
 
 \begin{proof} 
By definition,
 $
 \invprob_{k/2} \left[\Phi_a\right] =\invprob\left[\Phi_a^{(k/2)}(\ov\gamma)\right]$.
Hence \eqref{eq:localcpE} can be rewritten as
 \begin{equation}\label{eq:QQk2Q}
   \invprob\left[\Phi_a\right] \leq \invprob_{k/2} \left[\Phi_a\right]
   \leq  \invprob\left[\Phi_a\right] +   c \, e^{-k(d-1)/{6}}.\end{equation}
Similarly, the inequality above also holds if we replace $\invprob$ by $Q$. Moreover, by \eqref{eq:cstardef}, we see that $\invprob\left[\Phi_a\right] =c_* \,e^{-a\eta}$. 
Hence, it follows that
\[    
\left|Q[\Phi_a]-c_* \,e^{-a\eta} \right| \leq  
 c  e^{-k(d-1)/{6}} + d_{\op{TV}}\left(Q_{k/2} , \invprob_{k/2}\right)\sup_{\ov\gamma\in\wh\state_{k/2}} \Phi_a(\ov\gamma)  .\]

When $d=2$, it follows from Lemma \ref{dec8.lemma1} that $\sup_{\ov\gamma\in\wh\state_{k/2}} \Phi_a(\ov\gamma)<\infty$. When $d=3$, since $a\geq 1/10$, the finiteness of $\sup_{\ov\gamma\in\wh\state_{k/2}} \Phi_a(\ov\gamma)$ follows from the same type of argument as in the proof of Lemma \ref{dec8.lemma2} for the case of $d=3$. This finishes the proof of \eqref{eq101}. \end{proof}
The following spatially-inverted version of \eqref{eq:totalvariation}  (see Proposition \ref{invertprop} for more details on the inversion) controls the total variation distance on the right side of \eqref{eq101}.
  Suppose
$V_0,V_1,V_2 $ are pairwise disjoint subsets of  $\{|z| \geq e^{k}\}$  and $x_j \in V_j $, $j=1,2$, such that $V_0$, $V_0\cup V_1$, and $V_0\cup V_2$ are closed sets.We again refer to such quintuples $(V_0,V_1,V_2,x_1,x_2)$ colloquially as initial configurations. Let
 $W_t^1,W_t^2$ be independent Brownian motions starting at $x_1,x_2$ and,
 as before,  let $\Gamma^j =
 V_j \cup   W^j[0,T_{0}^j]$, $j=1,2$, where $T_0^j$ is the time at which $W^j$ first hits $\cB$. For $m<k$, let $Q_{m}$ denote
the probability measure on  $\state_{m}$ obtained by considering $W^j[T_m^j,T_{0}^j]$, $j=1,2$,conditioned on the following event\footnote{The conditioning on $\{T_0^1,T_0^2 < \infty\}$ is redundant if $d=2$.}
$$ \left\{(V_0\cup\Gamma^1)\cap \Gamma^2 =\Gamma^1\cap (V_0\cup\Gamma^2) = \eset \right\}\bigcap \left\{T_0^1,T_0^2 < \infty\right\}.$$
Again, to make sure our results make sense, we only consider initial configurations such that the event above has positive probability before conditioning.

\begin{proposition}   \label{prop102}
  There exist $ u,c > 0$  such that for 
all $(V_0,V_1,V_2,x_1,x_2)$ as above,  
\begin{equation}\label{eq:prop102}
 d_{\op{TV}}\left(Q_{k/2} , \invprob_{k/2}\right)\leq c e^{-uk}.
\end{equation}
\end{proposition}

Before we end this section we state without proof the following observation which will be implicitly recalled repeatedly in the next section.
\begin{observation}\label{obs:upper}
Suppose $z\notin \{x,y\}$ and $s>0$. 
Then there exists $C$ only depending on $|x-z|\wedge |y-z|$  such that \(\mu_{x,y}(E^s(z))\le C e^{-\eta s}\), where
$E^s(z)$ is the event that there is a cut point in $\cB_{-s}(z)$. 
\end{observation}

\section{Proofs}\label{sec:4}

\subsection{Proof of Theorem \ref{theorem.dec16.2} }
\label{sec:proofmaintheorem}

Let us fix $z\in\mathbb{R}^d\setminus\{0,\e\}$ and write
\[  \Lambda^0_k = \{\gamma \in \Gamma_{0,\e}:
  \dist (z,\gamma ) \leq e^{-k}  \}.
\]
Moreover, for $\gamma\in \Lambda^0_k$,  decompose $\gamma$ at first hitting of and  last exit from $\ball_{-k}(z)$  and write these truncated paths as $\gamma_1$ and $\gamma_2$. Then we write 
\[\Lambda_k=\{\gamma\in\Lambda^0_k\,:\, \gamma_1\cap \gamma_2=\emptyset\}\mbox{ and }\Theta_k = \{ \gamma \in \Gamma_{0,\e}: \dist(z,\A_\gamma )< e^{-k }\}.\]
Note that $\Lambda^0_k,\Lambda_k, \Theta_k$ decrease with $k$ and $\Theta_k
  \subset \Lambda_k\subset \Lambda^0_k$.
We let $S=\partial \ball_{-(b+k)}(z)$, {$D=\R^d\setminus\ball_{-(b+k)}(z)$} and write
\begin{equation}\label{eq:gkdefinition}
g_{k}=   \int_{S } \int_{S } \mu^D_{0,x} \times \mu_ {y,\e}^D \left[ 1_{\gamma_1\cap \gamma_2=\emptyset}\right]\sigma(dx,dy).
\end{equation}
To prove \eqref{eq:T1c1}, it
suffices to show that there exist positive $c_*, c',u$, and $m$ (all independent of $z$) such that
if $\dist(0,z,\e) \geq e^{-b}$, then for $k \ge m$ and $1 \leq a \leq 2$, 
\begin{equation}\label{eq4.1}
 \left|\mu[\Theta_{k+b+a}]/g_{k}
  - c_* \, e^{-a\eta}  \right|
   \leq c'\, e^{-uk}.
\end{equation} 
Indeed if we write $f_j = f_j(z) = \log \mu[J_j(z)]=\eta j+\log\mu[\Theta_j]$,
   then \eqref{eq4.1} implies that for $1 \leq a \leq 2$,
   $k \geq m$,
   \[        |f_{k+b+a} - f_{k+b+1}|   \leq c  \, e^{-uk} \]
   (both $c$ and all other constants below depend only on $c',u$, and $\eta$). 
 By summing we see this holds for all $a \geq 1$  for a different
 constant $c$ and by exponentiating we see that for $j \geq k \geq m+1$, 
 \[\mu[J_{j+b }] = \mu[J_{k + b}]\, [1+O(e^{-uk})],\]
 in other words, $\mu[J_k$] is a Cauchy sequence. This confirms \eqref{eq:T1c1}, hence \eqref{eq:G1def}.

We now prove ({\ref{eq4.1}). Let $B = \ball_{-b-k}(z)$.
Any path in $\Lambda^0_{b+k}$  can be written
as $\gamma = \gamma^1 \oplus \omega \oplus  [\gamma^2]^R $ where $\gamma^1$
is $\gamma$ stopped at 
the first visit to $S=\partial  B$,
which we denote by $x$; $[\gamma^2]^R$ is the
reversal of $\gamma^R$ stopped at the first visit to  $S$,
which we denote by $y$;
and $\omega$ is chosen from $\Gamma_{x,y}$.  In other words, $\mu$ restricted to $\Lambda^0_{b+k}$ can be written as
\begin{equation}\label{eq:decompexample}
      \int_{S } \int_{S } \left[\mu^D_{0,x} \oplus \mu_{x,y}
 \oplus \mu_ {y,\e}^D \right] \, \sigma(dx,dy),
\end{equation}
 where $D = \R^d \setminus B$.  

{Let $Q_D$ be the probability measure on a pair of paths $(\gamma_1,\gamma_2)$ given by
\begin{equation}
Q_D((\gamma_1,\gamma_2)\in \cdot) =  \frac{1}{g_k}\int_{S } \int_{S } \mu^D_{0,x} \times \mu_ {y,\e}^D \left[ 1_{\gamma_1\cap \gamma_2=\emptyset} (\gamma_1,\gamma_2)\in \cdot \right]\sigma(dx,dy),
\end{equation}
where $g_k$ is defined in \eqref{eq:gkdefinition}.
We let $Q$ stand for $Q_D$ after appropriate scaling and translation so that $Q$ fits} in the setting of Propositions \ref{prop101} and \ref{prop102}, where we let 
\begin{equation}\label{eq:quin}
(V_0,V_1,V_2,x_1,x_2)=(\eset,\{0\},\{\e\},0,\e).
\end{equation} Recall the definition of $\Phi_a$ before Proposition \ref{prop101}. It is not difficult to see that
$$
Q[\Phi_a]= \mu[\Theta_{k+b+a}]/g_{k}.
$$
Applying Propositions \ref{prop101}  and \ref{prop102}  to $Q$, we see that uniformly for $a\in[1,2]$, there exists $c_*, c, c',u>0$ such that 
\begin{equation}\label{eq:1.1final}
\left|Q[\Phi_a]-c_* \,e^{-a\eta} \right| \leq  
 c  \left[ e^{-k(d-1)/{6}} + d_{\op{TV}}\left(Q_{k/2} , \invprob_{k/2}\right)\right] \leq   c' e^{-uk}.
\end{equation}  
     
We now prove the second claim (\ref{eq:T1c2}). Consider $\mu_{0,\infty}$, the path measure between $0$ and infinity, defined in the following manner. 
More precisely, let $K>0$ and write \[\mu_{0,\infty}=\int_{S} \mu_{0,\zeta}\oplus \nu^{D}_{\zeta,\infty} \,  \sigma(d\zeta),\]
where $S=\{|z|=K \}$, $\sigma$ is the surface measure on $S$, 
and  $\nu^{D}_{\zeta,\infty}$ is the excursion measure on $\{|z|> K\}$ from $\zeta$ to infinity, defined as the limiting measure of Brownian motion from $(1+\epsilon)\zeta$ conditioned to avoid $\{|z|< K\}$ before hitting $\{|z|=L\}$, as $\epsilon\to 0$ and $L\to\infty$ (when $d=3$ we can consider the infinite Brownian motion conditioned to avoid $\{|z|< K\}$ directly), with total measure reweighed by $1/K$. It is easy to check that the definition is consistent for any choice of $K$ and that $\mu_{0,\infty}$ satisfies the same scaling property as $\mu_{0,\e}$ in \eqref{eq:scalingpty}. Note that this definition is mainly interesting for $d=2$, as in the case of $d=3$ the path measure $\mu_{0,\infty}$ is merely a constant multiple of the law of a Brownian motion started from 0.

Now, let $\phi$ be  the inversion  map on $\R^d$
with respect to the circle/sphere $\{|y-\e|=1\}$ and consider the path measure $\overline{\mu}_{0,\infty}$ defined by the pushforward of $\mu_{0,\e}$ by $\phi$, which is supported on paths from $0$ to infinity.  
By the inversion invariance of the trace of Brownian motion, the law of the trace (and hence of the set of cut points thereof) induced by $\overline{\mu}_{0,\infty}$ and $\mu_{0,\infty}$ are the same up to multiplication by a deterministic constant.

 Under this setup, the first claim \eqref{eq:T1c1} still follows. 
In this case, thanks to the scaling property and rotation invariance of $\mu_{0,\infty}$, we know that the corresponding cut-point Green's function is given by
\begin{equation}
\greencut_{0,\infty}(z)=c |z|^{-\eta}
\end{equation}
for some $c>0$. Taking into consideration the covariant derivative, we obtain \eqref{eq:T1c2}, the explicit formula of $\greencut(z)$.

\subsection{Decomposition of paths and up-to-constant two-point estimates}\label{sec:uptoc}
In this subsection, we will introduce various notations regarding path decomposition and prove a preliminary up-to-constant two-point estimate, which serves as an archetype of such argument. Variants of this estimate will appear more than once in the subsequent subsections.

  We fix $z,w\in V\in{\cal D}$, with $|z-w| = e^{-b}$.  Note that $|z|,|w|,|z-\e|,|w- \e|
   \geq 2 e^{-b}/\sqrt 3 $.   We will write 
  $B_1 = \ball_{-(b+3k)}(z),
B_1' =  \ball_{-(b+3k)}(w) ,B_2 = \ball_{-(b+3k+a)}(z),\\
B_2' =  \ball_{-(b+3k+a')}(w),  S = \partial B_1,
S' =\partial B_1'  $ and let $D$ denote the unbounded
domain bounded by $S$ and $S'$.   

Let $\newset=\{\gamma \in \Gamma_{0,\e}: \exists s<t\;\textrm{such that}\; \gamma(s)\in B_1,\gamma(t)\in B_1'\}$ and 
$\newset'=$ $\{\gamma \in \Gamma_{0,\e}: \exists s<t\;\textrm{such that}\; \gamma(s)\in B_1',\gamma(t)\in B_1\}$.
We can focus on $\newset$, which is sufficient due to symmetry. 

We start with the decomposition of paths. A path $\gamma \in \newset$
can be written as  
\begin{equation}\label{eq:gammadec}
\gamma = \gamma^1 \oplus \omega_1 \oplus \gamma^*
 \oplus \omega_2  \oplus  [\gamma^2]^R
\end{equation}
where
 \begin{itemize}
 \item  $\gamma^1$
is $\gamma$ stopped at 
the first visit to $S$; we denote the endpoint by $x$,
\item the reversal of $\gamma^2$ is $\gamma$ started at the last visit to $S'$;
 we denote the starting point by $y$, 
\item  $\gamma^*$ is an excursion in $D$  starting on $S$
ending on $S'$; we let $x',y'$
denote the initial and terminal vertices of $\gamma^*$, and
\item   $\omega_1 \in \Gamma_{x,x'}$,  
   $\omega_2 \in \Gamma_{y',y}$.
\end{itemize}
This decomposition is not necessarily unique. 
For a given $\gamma$, the 
number of such decompositions is the number of excursions from $S$
to $S'$ in $D$ that are contained in the path. However, if $\gamma \in \newset
\setminus \newset'$, the decomposition is indeed unique.
Later, we will see that paths in $\newset \cap \newset'$ that contain cut points in $B_2$ and $B'_2$ will form an exceptional
set of smaller measure, and are hence negligible.

In order to avoid the issue of non-uniqueness of the decomposition, we now consider a new measure $\wt\mu$ defined through concatenation of paths and work mainly with this measure. In a moment we will see that it will not differ too much from $\mu$ for our purpose in this subsection.

Let $\wt\mu$ be the measure on quintuples of paths $(\gamma^1,\omega_1 ,\gamma^*,\omega_2,\gamma^2)$, as well as (slightly abusing the notation) the induced measure on $\Gamma_{0,\e}$ of the concatenated path $\gamma = \gamma^1 \oplus \omega_1 \oplus \gamma^*
 \oplus \omega_2  \oplus  [\gamma^2]^R $, constructed as follows:
\begin{itemize}
\item  First choose $ \gamma^1,\gamma^2, \gamma^*$ independently from the measures:
\begin{itemize}
\item  Brownian motion started at $0$ stopped upon reaching  $S$  
(for $d=3$, restricted to the event that it reaches the ball), 
\item  Brownian motion started at $\e$ stopped upon reaching  $S'$  
(for $d=3$, restricted to the event that it reaches the ball), and
\item  the  excursion measure in $D$ from $S$ to $S'$,  i.e.,
\begin{equation*}
  \int_{S} \int_{S'}  \mu_{x',y'}^D \,   \sigma(d x') \, \sigma(d y'), 
\end{equation*} 
\end{itemize}
respectively. 
 \item  Given $(\gamma^1,\gamma^*,\gamma^2)$, and hence $(x,y,x',y')$, choose $\omega_1, \omega_2$
 independently from $\mu_{x,x'}$ and $\mu_{y',y}$, respectively.
\end{itemize}
Let $\cU$ be the set of quintuples of paths $(\gamma^1,\omega_1 ,\gamma^*,\omega_2,\gamma^2)$ such that after the concatenation  \eqref{eq:gammadec},
$\omega_1$ contains a cut point of $\gamma$ inside $B_2$ and
$\omega_2$ contains a cut point of $\gamma$ inside $B_2'$. Note that if $(\gamma^1,\omega_1 ,\gamma^*,\omega_2,\gamma^2)\in{\cU}$, then the event $N_0$ define in \eqref{eq:feb4.1} necessarily occurs.
Let $\wh \newset$ be the subset of $\newset$ induced from $\cU$, i.e.\ consisting of paths $\gamma$ concatenated from $(\gamma^1,\omega_1 ,\gamma^*,\omega_2,\gamma^2)\in{\cU}$ according to \eqref{eq:gammadec}. 
\begin{rmk}\label{rem:dom}
Restricted to $\newset $, $\mu$ is dominated by $\wt\mu$ (regarded as a measure on $\Gamma_{0,\e}$). To see this,  we decompose the paths in $\Gamma_{0,\e}$ in a unique way, such that $\gamma^*$ is the first excursion between $S$ and $S'$, and then observe that when comparing $\mu$ with $\wt\mu$ there are additional constraints on the path $\omega_1$ associated with $\mu$ since it is not allowed to make excursions from $S$ to $S'$. Moreover, it is not difficult to see  that $\mu$ agrees with $\wt\mu$ on $\newset\setminus\newset'$, the set of paths that has only one excursion between $S^1$ and $S^2$. 
\end{rmk}

We now construct a measure $\cP_0$ as follows: choose $ \gamma^1,\gamma^2, \gamma^*$ independently from the measures:
\begin{itemize}
\item  Brownian motion started at $0$ stopped upon reaching  $S$  
(for $d=3$, restricted to the event that it reaches the ball), 
\item  Brownian motion started at $\e$ stopped upon reaching  $S'$  
(for $d=3$, restricted to the event that it reaches the ball), and
\item  the  excursion measure in $D$ from $S$ to $S'$,  i.e.,
\begin{equation*}
  \int_{S} \int_{S'}  \mu_{x',y'}^D \,   \sigma(d x') \, \sigma(d y'), 
\end{equation*} 
\end{itemize}
(up to now everything is the same as the construction of $\wt\mu$ except we do not patch up with $\omega_1$ and $\omega_2$) and finally restrict to the event $N_0$
\begin{equation}\label{eq:feb4.1}
N_0:=\big\{ \gamma^1 \cap \gamma^* = \gamma^1 \cap \gamma^2 = \gamma^*
 \cap \gamma^2 = \eset  \big\}.
\end{equation}
From the definition we see that
\begin{equation}\label{eq:cPcard}
\|\cP_0\|= \int_{S'\times S' } \int_{S\times S }\mu^D_{0,x} \times \mu^D_{x',y'} \times \mu_ {y,\e}^D \left[ N_0\right]\sigma(dx,dx')\sigma(dy',dy).
\end{equation}
The next proposition gives an up-to-constant bound on $\|\cP_0\|$.
\begin{prop}\label{prop:cp0bound}
There exists a constant $c>0$ that depends on $V$ but not on $z,w$, such that
\begin{equation}\label{eq:cp0bound}
c^{-1} e^{-(b+6k)\eta}\leq \|\cP_0\|\leq c e^{-(b+6k)\eta}.
\end{equation}
\end{prop}

\begin{figure}
	\includegraphics[scale=1]{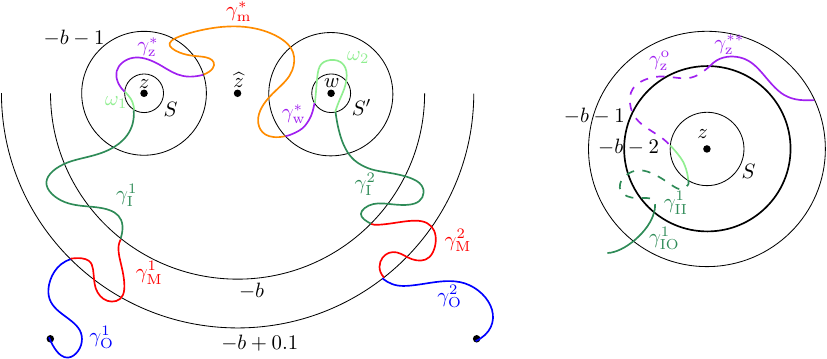}
	\caption{{Illustration of the proof of Proposition \ref{prop:cp0bound}. If we write $k$ next to a circle or circle arc, then this circle or circle arc has radius $e^{k}$.}}
	\label{fig:up-to-constant}
\end{figure}

\begin{proof}
{We start with decomposition of paths. See Figure \ref{fig:up-to-constant} for an illustration.}
By the definition of $\dyadic$ and basic geometric calculation, if $\wh z$ is the midpoint between $z$ and $w$ then $0,\e\notin \ball_{-b+0.1}(\wh z)$.
We further truncate $\gamma^1$ and $\gamma^2$ upon reaching the ball $\ball_{-b+0.1}(\wh z)$. More precisely, we let $\gamma^1_{\rm O}$ (resp.\ $\gamma^2_{\rm O}$) denote $\gamma^1$ (resp.\ $\gamma^2$) until reaching $\ball_{-b+0.1}(\wh z)$. Note that by the definition of $\dyadic$, we also truncate $\gamma^1$ from first entrance to $\ball_{-b}(\wh z)$ until it stops upon hitting $S$ and call it $\gamma^1_{\rm I}$.
Similarly, we truncate $\gamma^2$ from first entrance to $\ball_{-b}(\wh z)$ until it stops upon hitting $S'$ and and call it $\gamma^2_{\rm I}$. We call the middle pieces $\gamma^1_{\rm M}$ and $\gamma^2_{\rm M}$ such that $\gamma^1=\gamma^1_{\rm O}\oplus\gamma^1_{\rm M}\oplus \gamma^1_{\rm I}$ and $\gamma^2=\gamma^2_{\rm O}\oplus\gamma^2_{\rm M}\oplus \gamma^2_{\rm I}$.

We then consider the excursion $\gamma^*$. Write $\gamma^*=[\gamma^*_{\rm z}]^R\oplus \gamma^*_{\rm m}\oplus \gamma^*_{\rm w}$, where $[\gamma^*_{\rm z}]^R$ equals $\gamma^*$ from the beginning until the first exit from $\ball_{-b-1}(z)$, $\gamma^*_{\rm w}$ equals $\gamma^*$ from the last entry into $\ball_{-b-1}(w)$ until the end, and $\gamma^*_{\rm m}$ is the middle segment, which is a curve from $\partial\ball_{-b-1}(z)$ to $\partial\ball_{-b-1}(w)$. We further decompose $\gamma^*_{\rm z}$  at its first entry into $\ball_{-b-2}(z)$ as $\gamma^{**}_{\rm z}\oplus\gamma^\circ_{\rm z}$ and decompose $\gamma^*_{\rm w}$ as $\gamma^{**}_{\rm w}\oplus\gamma^\circ_{\rm w}$ similarly.

The probability of paths $\gamma^1_{\rm O},\gamma^2_{\rm O}$ reaching
$\ball_{-b+0.1}(\wh z)$ without intersection is bounded above by $O_V(e^{-b\eta})$, thanks to {\eqref{eq:uptocstproba}}. The measure of the set of paths $\gamma^1_{\rm I},\gamma^*_{\rm z}$ reaching $S$ with no intersection is bounded above by a constant multiple of $e^{-3k\eta}$, again thanks to {\eqref{eq:uptocstproba}}, and the same holds for paths $\gamma^2_{\rm I},\gamma^*_{\rm w}$ reaching $S'$ with no intersections. Note that here $\gamma^*_{\rm z}$ and $\gamma^*_{\rm w}$ are Brownian excursions, but $\gamma^\circ_{\rm z}$ conditioned on its initial point is a restriction of the probability measure given by a Brownian motion while the total mass of the measure governing $\gamma^{**}_{\rm z}$ is uniformly bounded and the same statement holds for $\gamma^*_{\rm w}$,  hence the non-intersection probability estimates remain valid. The implicit constants here (and below) satisfy the same dependencies as the constant $c$ in the statement of the lemma. 

To obtain the upper bound in \eqref{eq:cp0bound},
we first sample the truncated parts discussed above, then patch them up with the middle parts $\gamma^1_{\rm M}$, $\gamma^*_{\rm m}$, and $\gamma^2_{\rm M}$ whose total measure is uniformly bounded (the boundedness of the total mass of  $\gamma^*_{\rm m}$ comes from Lemma \ref{dec8.lemma1}, as the complement of the event $C$ in that lemma will {force $\gamma^*$ to intersect $\gamma^1$ and $\gamma^2$}). 

To show the lower bound of \eqref{eq:cp0bound}, when considering $\gamma^1_{\rm O}$ and $\gamma^2_{\rm O}$ we do not only restrict to the paths that are non-intersecting, but also that they are well-separated in the following sense: writing $v_1,v_2\in\partial \ball_{-b+0.1}(\wh z)$ for the ending points of $\gamma^1_{\rm O}$ and $\gamma^2_{\rm O}$, respectively, we say that  $\gamma^1_{\rm O}$ and $\gamma^2_{\rm O}$ are well-separated if
\begin{equation}\label{eq:wellsep1}
\dist(\gamma^1_{\rm O},v_2) \wedge \dist(\gamma^2_{\rm O},v_1) \geq e^{-b+0.1}/10.
\end{equation}
By the inverted version of the separation lemma \eqref{eq:sep1} and {\eqref{eq:uptocstproba}} the measure of paths $\gamma^1_{\rm O},\gamma^2_{\rm O}$ reaching
$\ball_{-b+0.1}(\wh z)$ without intersection and are well-separated is bounded from below by a constant multiple of $e^{-b\eta}$.

We now further decompose $\gamma^1_{\rm I}$ at its first entry into $\ball_{-b-2}(z)$ as $\gamma^1_{\rm IO}\oplus \gamma^1_{\rm II}$ and decompose $\gamma^2_{\rm I}$ as $\gamma^2_{\rm IO}\oplus \gamma^2_{\rm II}$ similarly. When considering $\gamma^1_{\rm II}$ and $\gamma^\circ_{\rm z}$ we do not only restrict to the event such that they are non-intersecting, but also that they are well-separated at both ends\footnote{The well-separatedness near $x,x'\in\partial S$ is not really necessary for this proposition, however it will be vital for the lower bound in the cases where we need to patch up $\omega_1$ and $\omega_2$.} in the following sense:
\begin{itemize}
\item letting $u_1,u_{\rm z}$ be the starting points of $\gamma^1_{\rm II}$ and $\gamma^\circ_{\rm z}$, respectively, and writing $e^{-b'}=e^{-b-2}/50$,
$$\gamma^1_{\rm II} \in \ball_{-b-2}(z)\cup \ball_{-b'}(u_1)\quad\mbox{ and }\quad\gamma^\circ_{\rm z} \in \ball_{-b-2}(z)\cup \ball_{-b'}(u_{\rm z});$$
\item recalling that $x,x'\in S$ are the ending points of $\gamma^1_{\rm II}$ and $\gamma^*_{\rm z}$, respectively,
\begin{equation}\label{eq:wellsep2}
\dist(\gamma^1_{\rm II},x')\wedge \dist(\gamma^\circ_{\rm z},x) \geq e^{-(b+3k)}/10.
\end{equation}
\end{itemize}
By the inverted version of the separation-at-both-ends lemma \eqref{eq:sep3} and {\eqref{eq:uptocstproba}}, assuming that $\dist(u_1,u_{\rm z})\geq e^{-b-2}/10$, 
the probability that $\gamma^1_{\rm II}$ and $\gamma^\circ_{\rm z}$ are non-intersecting and well-separated at both ends is bounded from below by a constant multiple of $e^{-3k\eta}$. We impose the same restriction on $\gamma^\circ_{\rm w}$ and $\gamma^2_{\rm II}$. The total measure of paths that satisfy theses restrictions are still bounded from below by a constant multiple of $e^{-b\eta} [e^{-3k\eta}]^2$. Now, when we patch up with $\gamma^1_{\rm M}$, $\gamma^1_{\rm IO}$, $\gamma^{**}_{\rm z}$, $\gamma^*_{\rm m}$, $\gamma^{**}_{\rm w}$, $\gamma^2_{\rm IO}$, and $\gamma^2_{\rm M}$,  we claim that the measure of paths such that $N_0$ occur is bounded from below by a universal positive constant. {To see this, we note that conditioned on other parts of the paths and on the event that they are well-separated as mentioned above, 
one can impose mild spatial restrictions to each path to be patched (i.e., forcing each of them to stay in a certain region) as to ensure $N_0$ still occurs; these regions can be chosen in a way such that the size of each region as well as the distance between the boundary of this region and the starting and ending points of the respective path are of the same order as the distance between the starting and ending points, where the constants are universal regardless of the conditioning, in particular the actual location of the starting and ending points; hence the measure of each path obeying its respective restriction is bounded from below by a universal positive constant regardless of their starting and ending points.}  This confirms the lower bound of \eqref{eq:cp0bound}.
\end{proof}

\subsection{Proof of the positivity of $c_*$ from \eqref{eq:cstardef}}\label{sec:cstar}
Before working on the proof of Theorem \ref{groundhog} we make a brief detour and prove the the positivity of $c_*$ from \eqref{eq:cstardef} as finally all ingredients are available and we can no longer delay its proof since this fact will be needed in the proof of Theorem \ref{groundhog}.

Recall notations above \eqref{eq:cstardef}. In this subsection, balls are centered at the origin if the center is omitted in the notation and $B(x,r)$ stands for a ball centered at $x$ of radius $r$. 

Note that to prove the positivity it suffices to consider the case of $k=0$. We are actually going to show a stronger result, namely that there is a positive $\mu$-measure such that $\gamma$ has a set of cut-points in $\ball_0$ of Hausdorff dimension at least $\delta=d-\eta$ (recall the notation from Section \ref{sec:intro}). Proofs of this type first appeared \cite[Proposition 2.2]{Law1996}, where cut times rather than cut points were considered. This seemingly small difference makes it difficult to directly infer the positivity of $c_*$ here. Hence we need to adapt the strategy, and derive a ``spatial'' version of \cite[Proposition 2.2]{Law1996} tailored to fit our setting in this work.

We start with the classical tool for proving lower bounds on Hausdorff dimension.
\begin{lemma}[Frostman's Lemma]\cite[Theorem 4.13]{Falconer}\label{lem:Frostman}
Suppose $X\subset\mathbb{R}^d$ and let $\nu$ be a positive measure supported on $X$ with $\nu(X)>0$. For $\wh\delta\in (0, d]$, let the $\wh\delta$-energy, $I_{\wh\delta}(\nu)$ be defined by
\begin{equation}\label{eq:Frostman}
I_{\wh\delta}(\nu):=\int_{\R^d}\int_{\R^d} |s-t|^{-\wh\delta} d\nu(s)d\nu(t).
\end{equation}
If $I_{\wh\delta}(\nu)<\infty$, then the Hausdorff dimension of $X$ is at least $\wh\delta$.
\end{lemma}
Before stating and proving the main ingredients, we need to introduce the following notion of ``well-separation'', similar to \eqref{eq:wellsep1} and \eqref{eq:wellsep2} from Section \ref{sec:uptoc}, which incorporates the inverted version of \eqref{eq:Deltadef}:
\begin{itemize}
\item we call $\bar\gamma=(\gamma^1,\gamma^2)\in\newstatetwo_0$ with ending points $x,y\in\partial\ball_0$ well-separated, if 
$$
\Delta':=\dist(\gamma^1,y)\wedge\dist(\gamma^2,x) \geq 1/10.
$$
\end{itemize}
It is not difficult to see that the claim of this subsection follows from the following lemma and proposition.
\begin{lemma} 
There exists $c>0$ such that 
$$
\invprob[ \Delta'\geq 1/10]\geq c.
$$
\end{lemma}
This lemma follows from an inverted version of the separation lemma \eqref{eq:sep1}.
\begin{prop}\label{prop:omegapatch}
There exists $c>0$ such that for all $x,y\in\partial\ball_0$ such that $|x-y|\geq 1/10$,
$$
\mu_{x,y} [ \;\omega \mbox{ is good and makes a set of cut points of dimension $\geq \delta$ in $\ball_{-1}$}\; ] \geq c,
$$
where $\omega$ is good {if
for $\omega=\omega^-\oplus\omega^*\oplus\omega^+$
with
\begin{itemize}
\item $\omega^-$ is $\omega$ stopped at the first entry of $\ball_{-1}$;
\item $\omega^*$ is $\omega$ chopped from the first entry of $\ball_{-1}$ to the last exit thereof;
\item $\omega^+$ is the rest of $\omega$ (i.e., from the last exit of $\ball_{-1}$ until the very end),
\end{itemize}
one has
\begin{itemize}
\item $\omega^-\subset \ball_0\cup B(x,1/50)$;
\item $\omega^* \subset \ball_{-0.5}$;
\item $\omega^+\subset \ball_0\cup B(y,1/50)$.
\end{itemize}}
\end{prop}
{Observe that if $\omega$ is good, then all cut points of $\omega$ in $\ball_{-1}$ are global cut points of $\gamma^1\oplus\omega\oplus\gamma^2$ for \emph{any} well-separated $(\gamma^1,\gamma^2)\in\newstatetwo_0$.}

%
\begin{proof}
We follow the same strategy as the proof of \cite[Proposition 2.2]{Law1996}, constructing a measure that satisfies the requirement of Lemma \ref{lem:Frostman}. The crux of the argument is to derive a lower bound on the one-point estimate and an upper bound on two-point estimate, i.e.,\ analogs of (4) and (6) of \cite{Law1996}.

We now fix $x,y\in\partial\ball_0$ such that $|x-y|\geq 1/10$. Note that none of the constants below depend on $x,y$. Recall the definition of the set of paths $C$ from Lemma \ref{dec8.lemma1}.  For $r<0$, write 
\begin{equation}\label{eq:Azrdef}
A(z,r):=\{ \omega \cap \ball_r(z)\neq \eset\}\cap \{\omega_1\cap\omega_3=\eset\}\cap \{\omega_2 \in C \}
\end{equation}
where we decompose $\omega=\omega_1\oplus\omega_2\oplus\omega_3$ as follows:
\begin{itemize}
\item $\omega_1$ is $\omega$ stopped at the first hitting of $\partial \ball_r(z)$;
\item $\omega_3$ is $\omega$ from the last hitting of $\partial\ball_r(z)$ until the end;
\item $\omega_2$ is the part of $\omega$ in between.
\end{itemize}

Let {$Z_r= \{x\in e^{-r}\mathbb{Z}^d : \ball_{-r}(z)\subset\ball_{-1}\}$}. 
We now define a random measure
$$
\nu_r(\omega):= e^{\delta r} \sum_{z\in Z_r} 1_{ A(z,r)}(\omega)U_{z,r} 
$$
where $U_{z,r}$ is the Lebesgue measure supported on $\ball_{r}(z)$  with total mass $1$. {We now claim that it suffices to show the following one- and two-point estimates:}

\begin{itemize}
\item There exists $c>0$ such that
\begin{equation}\label{eq:4im}
\mu_{x,y} \big[ \#\{z \in Z_r ;  A(z,r)\mbox{ holds}\} \big] \geq c e^{-\delta r}.
\end{equation}
This is the analogue of \cite[equation (4)]{Law1996}.

\item There exists $c'<\infty$ such that
\begin{equation}\label{eq:6im}
\mu_{x,y} \Big[ \big( \#\{z \in Z_r ;  A(z,r) \mbox{ holds}\} \big)^2 \Big]\leq c' e^{-2\delta r}. 
\end{equation}
This is the analogue of \cite[equation (6)]{Law1996}.

\end{itemize}
{By an argument similar to the proof of \cite[Proposition 2.2]{Law1996} (see below \cite[equation (7)]{Law1996}) \eqref{eq:4im} and \eqref{eq:6im} imply that there is an event of positive $\mu$-measure on which any subsequential weak limit of $\nu_r$ has finite $(\delta-\epsilon)$-energy (see \eqref{eq:Frostman} for the definition) for some $\epsilon>0$.}

Let $\nu$ denote {one of} the subsequential limits defined above. It is not difficult to see that on that event, $\nu$ is supported on the set of cut points of $\omega$ in $\ball_{-1}$ and admits no mass on $\omega$'s that are not good.

To prove \eqref{eq:6im}, it suffices to show that there exists $c''<\infty$ such that for all $z_1, z_2 \in Z_r$,
\begin{equation}\label{eq:5im}
\mu_{x,y} \big[ A(z_1,r) \cap A(z_2,r) \big] \leq c''  |z_1-z_2|^{-\eta} e^{-2r \eta}.
\end{equation}
This is the analogue of \cite[equation (5)]{Law1996}.

To see that \eqref{eq:4im} holds, it suffices to show that for any $z\in Z_r$, 
$$
\mu_{x,y}[ A(z,r)] \geq c e^{r \eta}
$$
where the constant $c$ does not depend on the choice of $x$, $y$, or $z$. Similarly to how we argue the lower bound in Proposition \ref{prop:cp0bound}, 
the probability that $\omega_1$ and $\omega_3$ arrive at $\ball_{r}(z)$ without intersection and well-separated at both ends is bounded from below by a constant multiple of $e^{r \eta}$ and the measure of $\omega_2$ such that $ A(z,r)$ holds given that $\omega_1$ and $\omega_3$ are well-separated is uniformly bounded from below by a universal constant. 

The universality of the constant $c$ comes from the following observation: decomposing $\omega_1$ and $[\omega_3]^R$ at first entry of $\ball_{-0.5}(z)$,  then the probability of the inner parts to hit $\ball_{r}(z)$ non-intersecting and well-separated \`a la the separation-at-both-ends lemma is bounded below by $c_0e^{r \eta}$ where $c_0$ is universal given that the starting points are well-separated, and in attaching the outer parts the total measure of paths that do not destroy the non-intersection is bounded below by a universal constant (note that $\ball_0$ and $\ball_{-0.5}(z)$ are ``on the same scale''), thanks to the separation at the beginning for the inner parts. 

We now turn to \eqref{eq:5im}. Similar to the definition of $\newset$ and $\newset'$, a priori we need to decompose $\omega$ in two ways, but thanks to the symmetry it suffices to consider one:
consider the decomposition of $\omega$ as $\omega=g_1\oplus o_1 \oplus g^* \oplus o_2 \oplus [g_2]^R$ similarly as in \eqref{eq:gammadec} but with $S$ and $S'$ replaced by $\partial \ball_{r}(z_1)$ and $\partial \ball_{r}(z_2)$, respectively.
Define a new measure $\wt\mu_{x,y}$ on the quintuple $(g_1,o_1,g^*,o_2,g_2)$ as well as the induced measure on $\Gamma_{x,y}$. 
By Remark \ref{rem:dom}, to obtain the upper bound in \eqref{eq:5im}, it suffices to obtain an upper bound on
$
\wt\mu_{x,y} [A(z_1,r) \cap A(z_2,r) ],
$
which follows from an argument similar to that of \eqref{eq:cp0bound}, with the extra step of patching up $o_1$ and $o_2$, where we see by the definition of $A(\cdot,r)$ in \eqref{eq:Azrdef} that $o_1$ and $o_2$ must abide the event $C$ defined in Lemma \ref{dec8.lemma1}, hence by the same lemma their total measure is bounded above by a universal constant, regardless of their starting and ending points. This finishes the proof of the proposition.
\end{proof}

\subsection{Proof of Theorem \ref{groundhog}}   
Recall all notations from Section \ref{sec:uptoc} where we introduced the decomposition of paths and the new measure $\wt\mu$. As we are going to see in a moment, the crux of the proof is to derive precise asymptotics for the total mass of $\wh\newset$ under $\mu$. However, as a first step, we need to give rough up-to-constant asymptotics, as well as assert that exceptional paths 
are truly exceptional in that their total mass is exponentially smaller than that of $\wh\newset$. We will phrase and prove these results under $\wt\mu$, which is easier to work with.

\begin{lemma}  \label{lemma1}  There exist $0<c\leq c'<\infty$ (which may depend on $V$ but are otherwise independent of $z$ and $w$) such that
\begin{equation}  \label{jan26.1}
c \, e^{-(b+6k) \eta} \leq \wt\mu[{\cU}]
   \leq c' \,  e^{-(b+6k) \eta}.
   \end{equation}
 Moreover, there exists $u > 0$ such that the $\wt\mu$-measure of
quintuples of paths $(\gamma^1,\omega_1,\gamma^*,\omega_2,\gamma^2) \in {\cU}$ for which at least one of the following four events occur
is less than $c \, e^{-(b+6k)\eta - k/2}$: 
\begin{equation}\label{eq:exception}
\gamma^1 \cap \ball_{-(b+k)}(w) \neq \eset\mbox{, } \gamma^2 \cap \ball_{-(b+k)}(z) \neq \eset\mbox{, } \omega_1 \not\subset \ball_{-(b+ 2k)}(z),\mbox{ and }
\omega_2 \not\subset\ball_{-(b+ 2k)}(w).
\end{equation}
\end{lemma}

\begin{proof}
The proof of \eqref{jan26.1} mostly follows from the same strategy as in Proposition \ref{prop:cp0bound}, with the following modifications.
\begin{itemize}
\item In the upper bound, we still need to patch up $\omega_1$ and $\omega_2$  after sampling the triple $(\gamma^1,\gamma^*,\gamma^2)$. Note that in this process by Lemma \ref{dec8.lemma1} for $d=2$ and the argument for uniform finiteness of $\Phi_a$ in the proof of Proposition \ref{prop101} for the case of $d=3$,  the total measure of $\omega_1$ and $\omega_2$ such that $\gamma\in \wh{\cal V}$ is uniformly bounded regardless of $\gamma^1$, $\gamma^*$, and $\gamma^2$.

\item In the lower bound, note that by the well-separatedness of $\gamma^1_{\rm II}$ and $\gamma^\circ_{\rm z}$ near $x,x'\in S$  as well as that of  $\gamma^\circ_{\rm w}$ and $\gamma^2_{\rm II}$ near $y',y\in S'$, we can patch up $\omega_1$ and $\omega_2$ as in Proposition \ref{prop:omegapatch} to ensure $\gamma\in \wh{\cal V}$ and to guarantee the positivity of $c$ from \eqref{jan26.1}.
\end{itemize}

To show the second claim we perform the same decomposition as above. For the first inequality, observe that 
\begin{itemize}
\item $\gamma^1_{\rm O}\oplus\gamma^1_{\rm M}$ cannot enter $\ball_{-(b+k)}(w)$, hence it suffices to consider $\gamma^1_{\rm I}$;
\item if $d = 3$ the measure of paths $\gamma^1_{\rm I}$ that hit
$\ball_{-(b+k)}(w)$ is comparable to $e^{-k}$, and if
$d=2$, by a Beurling estimate, the measure of paths reaching $\ball_{-(b+k)}(w)$  without intersecting $\gamma^*_{\rm w}$ is comparable to $e^{-k/2}$.
\end{itemize} 
 If $\gamma^1_{\rm I}\cap \ball_{-(b+k)}(w)\neq \emptyset$, then write $\gamma^1_{\rm tail}$ for the first entry of $\gamma^1_{\rm I}$ into $\ball_{-b-1}(z)$ after hitting $ \ball_{-(b+k)}(w)$. Then the first inequality follows from an argument similar to  the upper bound in \eqref{jan26.1} where we replace $\gamma^1_{\rm I}$ by $\gamma^1_{\rm tail}$. A similar bound works for the second inequality. 
 
 The third and fourth inequalities follow from Lemma \ref{dec8.lemma2}. The factor $e^{-k/2}$ can be changed to $e^{-k}$ if $d=3$ but we do not need the
stronger estimate.
\end{proof}

We now decompose the paths in another manner as follows:
 \[   \gamma^1 =  \wh \gamma^1 \oplus \wt \gamma^1, \;\;\;\;
   \gamma^2 = \wh \gamma^2 \oplus \wt \gamma^2, \;\;\;\;
     \gamma^* =  [\gamma_-^*]^R \oplus \wh \gamma^*
      \oplus  \gamma_+^* , \]
 where (the text in brackets refers to its illustration in Figure \ref{fig2})
 \begin{itemize}
 \item $\wh \gamma^1$ (blue) is $\gamma^1$ stopped at the first visit to
 $\partial \ball _{-(b+k)}(z)$; 
 \item  $\wt \gamma^1$ (red, normal+broad brush) is $\gamma^1$ started at the first visit to
 $\partial \ball _{-(b+k)}(z)$;
 \item $\wh \gamma^2$  (blue) is  $ \gamma^2$
stopped at the first visit to
 $\partial \ball _{-(b+k)}(w)$;
 \item  $\wt \gamma^2$ (red, normal+broad brush) is   $ \gamma^2$
 started at the first visit to
 $\partial \ball _{-(b+k)}(w)$;
 \item  $ [\gamma_-^*]^R$ (red, normal+{\bf broad} brush) is $\gamma^*$ stopped at the last visit to $\partial \ball _{-(b+k)}(z)$ before the first visit to $\partial \ball _{-(b+k)}(w)$, or, in other words,   $\gamma_-^*$ is $[\gamma^*]^\rev$ started at the first visit to $\partial \ball _{-(b+k)}(z)$ after the last visit to $\partial \ball _{-(b+k)}(w)$;
 \item $\gamma^*_+$ (red, normal+broad brush) is  $\gamma^* $ started at the first visit to
 $\partial \ball _{-(b+k)}(w)$ (note that the definition of $\gamma^*_+$ and $\gamma^*_-$ is not symmetric); 
 \item $\wh \gamma^*$ (blue) is an excursion from $\partial \ball _{-(b+k)}(z)$
 to $\partial \ball _{-(b+k)}(w)$ in $V$, the unbounded domain  whose
 boundary is 
 $\partial \ball _{-(b+k)}(z)  \cup \partial \ball _{-(b+k)}(w)$.
 \end{itemize} 
 By an argument similar to the proof of \eqref{eq:exception}, one sees that for some $C>0$ that only depends on $V$,
\begin{equation}\label{eq:lastrestriction}
\wt\mu
\left[ {\cU} \cap \{\gamma^*_+\cap \ball_{-(b+k)}(z)\neq\eset\}\right] \leq C e^{-(b+3k)\eta-k/2}.
\end{equation}

\begin{figure}[h]
\centering
\includegraphics[scale=0.88]{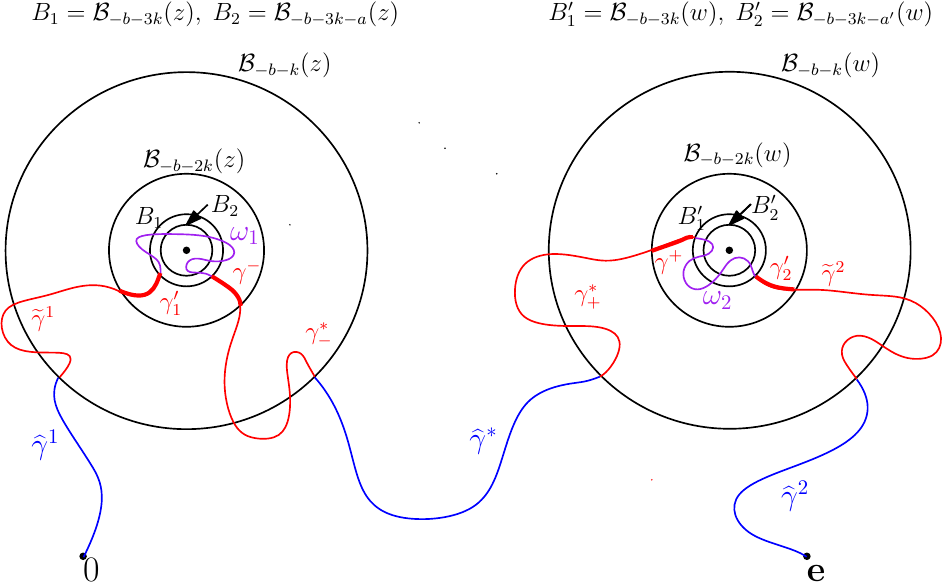}
\caption{\small{A typical sample from $\wt\mu$ restricted to ${\cU}$. Note that except on a set with very small measure, $\omega_1$ and $\omega_2$ stay in ${\cal B}_{-b-2k}(z)$ and ${\cal B}_{-b-2k}(w)$, respectively. As in the proof of Theorem \ref{theorem.dec16.2}, this implies that the cut points of $\gamma_1'\oplus\omega_1\oplus[\gamma^-]^\rev$ and $\gamma^+\oplus\omega_2\oplus[\gamma'_2]^\rev$ are also global cut points of $\gamma$. This allows us to compare the measure in (\ref{eq:2pt}) with two independent copies of the invariant measure for nonintersecting Brownian motions.}
}
\label{fig2}
\end{figure} 

We let $ \bar \gamma_z  =(\gamma_1',\gamma^-)$  be the pair of paths (both in broad red brush in Figure \ref{fig2}) obtained by discarding the part of $\wt \gamma^1$ and $\wt \gamma^*_-$ before their respective first visit to $\partial \ball_{-(b+2k)}(z)$ and define $\bar \gamma_w=(\gamma^+,\gamma_2')$ similarly. 


We now state and prove the two-point coupling result, which, along with Propositions \ref{prop101} and \ref{prop102}, proves (\ref{eq:2pt}) below and in turn (\ref{eq:T2c1}).
Let $\overline{\invprob}_k$ denote the probability measure on $(\bar \gamma_z, \bar \gamma_w)$  induced by $\wt\mu$ restricted to the event $N:=N_0 \cap N_1 \cap N_2\cap N_3$
 where $N_0$ is defined in \eqref{eq:feb4.1} while
 $N_1:=\{ \gamma_1\cap \ball_{-(b+k)}(w)=\eset\}$, $N_2:= \{\gamma_2\cap \ball_{-(b+k)}(z)=\eset\}$, and $N_3:=\{\gamma^*_+ \cap \ball_{-(b+k)}(z)=\eset\}$  are the events that appeared in \eqref{eq:exception} and \eqref{eq:lastrestriction}. Slightly abusing notation we also regard $\overline{\invprob}_k$ as a measure on a pair of properly rescaled and translated paths so that they fit in the setting of Propositions \ref{prop101} and \ref{prop102}.
 
 \begin{proposition}\label{prop4.2}   The total variation distance between
 $\overline{\invprob}_k$ and $\invprob_k \times \invprob_k$ (see the paragraph below (\ref{eq:cstardef}) for definition) is less than $c e^{-uk}$ where  $u$ is universal and $c$ may depend on $V$.
 \end{proposition}

\begin{proof}  
We write $\cP$ for $\cP_0$ restricted to event $N$. Similar to Proposition \ref{prop:cp0bound}, we have the following up-to-constant estimate
\begin{equation}\label{eq:cpbound}
c^{-1} e^{-(b+6k)\eta}\leq \|\cP\|\leq c e^{-(b+6k)\eta}.
\end{equation}
By an argument similar to the proof of \eqref{eq:exception}, 
we have the following control: 
$$
\cP[ N_4^c ]=O_V(e^{-uk})e^{-(b+6k)\eta},
$$
where
$$
N_4:=   \{   \wt \gamma^1 \cup \gamma_-^* \subset  \ball_{-(b+1)}(z)\}\cap 
   \{ \wt \gamma^2 \cup \gamma_+^* \subset  \ball_{-(b+1)}(w) \}.
$$
Observe that $\cP$ normalized to be a probability measure is equal to the measure $\overline{\invprob}_k$ in the statement of the proposition.
Therefore, denoting by $\overline{\invprob}'_k$ the law of $(\bar\gamma_z,\bar\gamma_w)$ induced by $\cP$ restricted to $N_4$, it follows that
$$
d_{\op{TV}}\left(\overline{\invprob}_k,\overline{\invprob}'_k\right)=O_V(e^{-uk}).
$$
Finally (regarding $\overline{\invprob}'_k$ as a measure on a pair of rescaled and translated paths similarly as how we regard $\overline{\invprob}_k$), we claim that it is possible to couple  
$\overline{\invprob}'_k$ with two independent 
	versions of the ``one-point''
invariant measure for non-intersecting Brownian motions from Section \ref{sec:3}, which yields
\begin{equation}\label{eq:twocoupling}
d_{\op{TV}}\left(\overline{\invprob}'_k,\invprob_k \times \invprob_k
\right)=O_V(e^{-uk}).
\end{equation}
To see this, we observe that from the restriction and conditioning above, if we condition on $(\wh\gamma^1,\wh\gamma^*,\wh\gamma^2)$ the law of $(\wt\gamma^1,\gamma^*_-)$ is a pair of Brownian motion conditioned on the event that
$$
\big(V_1\cup\wt\gamma^1\big)\cap \big(V_0\cup V_2\cup \gamma^*_-\big)=\big(V_2\cup\gamma^*_-\big) \cap \big(V_0\cup V_1\cup \wt\gamma^1\big)=\eset,
$$
where $V_0=\{z\in\mathbb{R}^d; |z|\geq e^{-(b+1)}\}$, $V_1=\ball_{-(b+1)}(z)\cap \wh\gamma^1$, and $V_2=\ball_{-(b+1)}(z)\cap \big(\wh\gamma^*\cup\wh\gamma^2\big)$, which is independent of $(\gamma^*_+,\wt\gamma^2)$. A similar observation follows for $(\gamma^*_+,\wt\gamma^2)$.
This allows us to independently apply Proposition \ref{prop102} twice to obtain \eqref{eq:twocoupling}, thus finishing the proof.
\end{proof}
\begin{proof}[Proof of Theorem \ref{groundhog}]
We start with \eqref{eq:T2c1}. First we claim that it suffices to prove that  there exists $m>0$ such that for all $k\geq m$, $1\leq a,a',a'',a'''\leq 2$,
\begin{equation}\label{eq:2pt}
\mu\left[J_{b+3k+a}(z)J_{b+3k+a'}(w)\right]\simeq \mu\left[J_{b+3k+a''}(z)J_{b+3k+a'''}(w)\right],
\end{equation}
where we write $A\simeq B$ for $A=B\big[1+O_V(e^{-uk})\big]$ for some  universal $u>0$ and the constant in $O_V(\cdot)$ may depend on $V$ but not on other variables. The claim follows since we obtain \eqref{eq:T2c1} by a summing argument similar that for \eqref{eq:T1c1} in Section \ref{sec:proofmaintheorem}.

Recall the definition of $\newset$ and $\newset'$ {in Section \ref{sec:uptoc}} and the definition of $\Theta_k$ at the beginning of Section \ref{sec:proofmaintheorem} and define $\Theta_k'$ similarly, replacing $z$ by $w$. We then write $\Upsilon_{a,a'}:=\Theta_{b+3k+a}\cap \Theta'_{b+3k+a'}$ for short.
We remark that the left side of \eqref{eq:2pt} can be rewritten as 
\begin{equation}\label{eq:JintoUpsilon}
\mu[J_{b+3k+a}(z)J_{b+3k+a'}(w)]=e^{(2b+6k+a+a')\eta}\mu[\Upsilon_{a,a'}].
\end{equation}

We now decompose $\Upsilon_{a,a'}$. Let  $\Xi,\Xi' \subset\Upsilon_{a,a'}$ be the paths that have a cut point in $B_2$ before having a cut point in $B'_2$ and vice versa, respectively. In other words, a path is contained in $\Xi$ if and only if there are times $t<s$ such that the path makes a cut point in $B_2$ at time $t$ and a cut point in $B'_2$ at time $s$. Note that $\Xi\cap\Xi'\neq\eset$ and $\Xi\cup\Xi'=\Upsilon_{a,a'}$.
Recall that $\newset \cup \newset'$ can be partitioned into $\newset \setminus \newset'$, $\newset \cap \newset'$, and $\newset' \setminus \newset$. 
By the definition of the various sets,
$$(\newset\setminus\newset')\cap \Upsilon_{a,a'}=(\newset\setminus\newset')\cap \Xi,\qquad (\newset'\setminus\newset)\cap \Upsilon_{a,a'}=(\newset'\setminus\newset)\cap \Xi'.$$ 
Furthermore,
\begin{equation}\label{eq:mubound}
\mu[(\newset\setminus\newset')\cap \Xi]+\mu[(\newset'\setminus\newset)\cap \Xi']\leq \mu[\Upsilon_{a,a'}] \leq \mu[\Xi]+\mu[\Xi'].
\end{equation}

We now turn to $\Xi$ (similar results follow for $\Xi'$ by symmetry). It follows from the definition of $\wh\newset$ and the dominance of $\wt\mu$ over $\mu$ on $\newset$ (see Remark \ref{rem:dom}) that
\begin{equation}\label{eq:Xi1}
\mu[(\newset\cap \newset')\cap \Xi] \leq \wt\mu[(\newset\cap \newset')\cap \Xi]\leq \wt\mu[{\cU}\cap E],
\end{equation}
where $E$ is the union of the four events in \eqref{eq:exception} {and the event on the left side of \eqref{eq:lastrestriction}}. For the second inequality, note that although an element in $(\newset\cap \newset')\cap \Xi$ may have more than one decomposition, any of the decompositions must be in ${\cU}\cap E$. On the one hand, by Lemma \ref{lemma1}, we have 
that for some universal $u>0$,
\begin{equation}\label{eq:Xi2}
\wt\mu[{\cU}\cap E]=O(e^{-uk})\wt\mu[{\cU}].
\end{equation}
On the other hand, by the definition of $E$, $\cU$, and Remark \ref{rem:dom}, we have
\begin{equation}\label{eq:Xi3}
\wt\mu[{\cU}\setminus E]\leq \wt\mu[(\newset\setminus\newset')\cap\Xi]=\mu[(\newset\setminus\newset')\cap\Xi]\leq \wt\mu[{\cU}].
\end{equation}
By \eqref{eq:Xi1}, \eqref{eq:Xi2} and \eqref{eq:Xi3}, we obtain
\begin{equation}\label{eq:cUbound}
\mu[(\newset\setminus\newset')\cap\Xi]\simeq  \wt\mu[{\cU}]\simeq \mu[\Xi].
\end{equation}
Combining this estimate and \eqref{eq:mubound} and using symmetry in $z$ and $w$ we get
\eqbn
\mu[\Xi]+\mu[\Xi'] \simeq \mu[\Upsilon_{a,a'}].
\eqen
Also using \eqref{eq:JintoUpsilon} and symmetry in $z$ and $w$ again, we see that to prove \eqref{eq:2pt}, it suffices to investigate the asymptotic behavior of $\mu[\Xi]$, or equivalently that of $\wt\mu[\cU]$, thanks to \eqref{eq:cUbound}. Also using that $\|\cP\|$ is independent of $a$ and $a'$, we see that it suffices to prove
\begin{equation}\label{eq:2ptasympt}
\wt\mu\left[{\cU}\right]
  \simeq  c_*^2 e^{-(a+a')\eta}\|\cP\|
\end{equation} 
where $c_*$ comes from \eqref{eq101} and $\cP$ from the proof of Proposition \ref{prop4.2}.

We now prove \eqref{eq:2ptasympt}.
Let
$$
M_1:=\{\mbox{$\omega_1$ has a cut point of $\gamma'_1\oplus \omega_1\oplus [\gamma^-]^R$ in $B_2$}\}\cap\{\omega_1 \in \ball_{-(b+2k)}(z)\}
$$
and
$$
M_2:=\{\mbox{$\omega_2$ has a cut point of $\gamma^+\oplus \omega_2\oplus [\gamma'_2]^R$ in $B'_2$}\}\cap\{\omega_2 \in \ball_{-(b+2k)}(w)\}.
$$

By the definition of $E$, 
\begin{equation}\label{eq:NM1M2}
\wt\mu\left[{\cU}\setminus E\right]\leq\wt\mu[  N \cap M_1 \cap M_2]\leq \wt\mu[\cU].
\end{equation}
Noting that by \eqref{eq:Xi2},  $\wt\mu\left[{\cU}\right]\simeq \wt\mu\left[{\cU}\setminus E\right],$ we see that \eqref{eq:NM1M2} implies
$$
\wt\mu[\cU]\simeq\wt\mu[  N \cap M_1 \cap M_2].
$$
By the definition of $\cP$ in Proposition \ref{prop4.2},
$$
 \wt\mu[  N \cap M_1 \cap M_2   ]= \cP\Big[\mu_{x,y}[M_1]\mu_{x',y'
 }[M_2]\Big].
$$
Observe that under $\cP$, $M_1$ and $M_2$ depend on $\bar\gamma_z$ and $\bar\gamma_w$ only. Recall the definition of $\Phi$ and $\wh\Phi$ in \eqref{eq:Phidef}, and the definition of $\overline{\invprob}_k$  before Proposition \ref{prop4.2} (note that $\overline{\invprob}_k=\cP/\|\cP\|$) we see that
$$
\cP\Big[\mu_{x,y}[M_1]\mu_{x',y' }[M_2]\Big]=\overline{\invprob}_k\Big[\wh\Phi_{a,k}(\bar\gamma_z)  \wh\Phi_{a',k}(\bar\gamma_w)\Big]\|\cP\|.
$$
Similar to \eqref{eq:localcp} and \eqref{eq:localcpE}, again by Lemma \ref{dec8.lemma2}, we have
$$
\overline{\invprob}_k\Big[\wh\Phi_{a,k}(\bar\gamma_z)  \wh\Phi_{a',k}(\bar\gamma_w)\Big]\simeq \overline{\invprob}_k\Big[\Phi_{a}(\bar\gamma_z)  \Phi_{a'}(\bar\gamma_w)\Big].
$$

Finally, by Propositions \ref{prop4.2} and \ref{prop101}, it follows that 
$$
\overline{\invprob}_k\Big[\Phi_{a}(\bar\gamma_z)  \Phi_{a'}(\bar\gamma_w)\Big] \simeq {\invprob}_k[\Phi_a]{\invprob}_k[\Phi_{a'}]\simeq c_*^2 e^{-(a+a')\eta}.
$$
This confirms \eqref{eq:2pt} as well as \eqref{eq:T2c1}.

To prove the second claim (\ref{eq:T2c2}) 
 it suffices to show that uniformly in $s$,
\begin{equation}\label{eq4.6}
c_V  |z-w|^{-\eta }\leq\mu\Big[J_s(z)J_{s+\rho}(w) \Big] \leq C_V  |z-w|^{-\eta }.
\end{equation} 
By \eqref{eq:JintoUpsilon}, \eqref{eq:mubound} and \eqref{eq:cUbound}, the inequality \eqref{eq4.6} readily follows from the asymptotics of $\mu[\Xi]$ and $\mu[\Xi']$, which ultimately comes from \eqref{jan26.1} along with a similar bound with $z$ and $w$ interchanged.
\end{proof}

\subsection{Existence of Minkowski content}  \label{gensec}
In this subsection, we first give a general proof of the existence of Minkowski content 
given the sharp Green's function estimates and some mild assumptions. Then we apply this general result to Brownian cut points and show Theorem \ref{theorem.dec16}.

Suppose $\mu$ is a $\sigma$-finite measure on compact subsets $\A$ of $\R^d$ and that $K\subset\R^d$
is a compact set (in our case $K = \{0,\e\}$). 
  We will
consider conditions under which there exists
a Borel measure $\nu = \nu_\A$ such that
$\nu(K) = 0$ and 
for all dyadic cubes $V$ with $V \cap K = \eset$,
\[      \nu(V) = \Cont_\delta(V \cap \A).\]
Here, $0 < \delta < d$ and we write
$\eta = d - \delta$ (as in earlier sections).
 We write $\mu[X]$ for the integral
of $X$ with respect to $\mu$. 

 \begin{proposition}\label{prop:generalMC}
Suppose $\mu$ is a $\sigma$-finite measure on compact subsets $\A$ of $\R^d$ and that $K\subset\R^d$ is compact.  Suppose that for every {bounded Borel} set $V\subset\R^d$ with $V \cap K = \eset$, we have
\begin{equation}\label{eq4.3}       
\mu\left[\{\A: \A \cap V \neq \eset \} \right]< \infty . 
\end{equation}  
Suppose $0 < \delta = d-\eta < d$.  Suppose $U,V$ are {bounded Borel} subsets of $\R^d$
disjoint from $K$ with $V \subset {\rm int}(U)$ and
 such that for some $\epsilon > 0$,
\begin{equation}\label{eq:boundarycond}
\Cont_{d- \epsilon}(\partial V)  = 0 .
\end{equation}   
Let $I_s(z) =  1_{\{\dist(z,\A)  \leq e^{-s} \}},$
$ J_s(z) = e^{\eta s} \,I_s(z)$, and 
\[  J_{s,V} = \int_V J_s(z) \, dz.\]
Suppose   the following holds for $z \neq w \in U$.
\begin{itemize}
\item The limits
\[  G(z) = \lim_{s \rightarrow \infty} \mu\left[J_s(z) 
\right], \mbox{ and }\]
\[  G(z,w) = \lim_{s \rightarrow \infty}
    \mu \left[J_s(z) \, J_s(w)\right], \]
exist and are finite.  Moreover, 
$G$ is uniformly bounded on $U$. 
\item  There exist  $c, \rho_0,u >0$ such that
 for $ 0 \leq \rho \leq \rho_0$,
\begin{equation}
 \label{dec15.1}
 \begin{split}
 &G(z) \leq c; \;\; G(z,w) \leq  c \, |z-w|^{-\eta};
   \;\;   \left| \mu[J_s(z)] - G(z)\right| \leq c \, e^{-us};\\
 &\left|\mu[J_s(z) \, J_{s+\rho}(w)] -
 G(z,w)\right| \leq  c \, |z-w|^{-\eta} \, e^{-us}.
 \end{split}
  \end{equation}
 \end{itemize}
 Then the finite limit
$ J_V =  \lim_{s \rightarrow \infty} J_{s,V} $
exists both
almost $\mu$-everywhere and in $L^2$. {In the meanwhile,\begin{equation}\label{eq:boundaryV}
\Cont_{\delta}(\partial V\cap \A)=0\quad\mbox{ $\mu$-a.s.}
\end{equation}}
 Moreover, one has
\begin{equation}\label{eq:JV}
J_V = \Cont_\delta(V \cap \A),
\end{equation}
\begin{equation}    \mu\left[J_V \right] =  \lim_{s \rightarrow \infty}
   \mu\left[J_{s,V} \right] = \int_V G(z) \, dz, \mbox{ and}
\end{equation}
\begin{equation}    \mu\left[J_V^2 \right] =  \lim_{s \rightarrow \infty}
   \mu\left[J_{s,V}^2 \right] = \int_V  \int_V G(z,w) \, dz \, dw.
\end{equation}

 \end{proposition}
 
\begin{proof} We fix $U,V$ as in the statement of the proposition,  write $J_s = J_{s,V}$,
and allow constants to depend on $U,V$. The result
is the same if we restrict to $\wt \mu = \mu \, 1_{\{
\A \cap U \neq \eset\}}$, which is $\mu$-almost surely a finite measure by \eqref{eq4.3}. Hence we can
normalize this to make it a probability measure.  So without
loss of generality we will assume that $\mu$ is a probability
measure and use $\Prob,\E$ notation.  We
assume $  0 \leq \rho \leq \rho_0$.
Note that \eqref{dec15.1} implies that  
\begin{equation}
\left| \mu\left[(J_{s+\rho}(z) - J_s(z)) \, (J_{s+\rho}(w)
  - J_s(w)) \right] \right| \leq  c \, |z-w|^{-\eta} \, e^{-us}.
\end{equation}
Since
\begin{eqnarray*}
    \left(J_{s+\rho} - J_s\right)^2  
  & = & 
\int_V \int _V  [J_{s+\rho}(z) - J_z(s)] \,
  [J_{s+\rho}(w) - J_s(w)]\, dw\, dz ,
\end{eqnarray*} 
we see   that
\begin{equation}
\E\left[(J_{s+\rho} - J_s )^2  \right] \leq c  \, e^{-us} 
\int_V \int_V  |z-w|^{-\eta}\, dw \, dz \leq c \, e^{-us}.
\end{equation}    
Using this we can see that for every $0 < \rho \leq \rho_0$, the sequence
$J_{n\rho}$ converges in both $L^2$ and $\mu$-almost everywhere to a finite limit
$J_{V,\rho}$.  Also, since
\[                 e^{\rho \eta} J_{n\rho} \geq J_s \geq e^{-\rho \eta}
 \, J_{(n+1) \rho}, \;\;\;\; n \rho \leq s \leq (n+1) \rho , \]
we can let $\rho \rightarrow 0$ and prove that $J_s \rightarrow J_V$
almost everywhere and in $L^2$.  Since the convergence is in $L^2$, we have
\[    \E\left[J_V \right] =  \lim_{s \rightarrow \infty}
   \E\left[J_{s,V} \right] = \int_V G(z) \, dz,\mbox{ and}\]
   \[    \E\left[J_V^2 \right] =  \lim_{s \rightarrow \infty}
   \E\left[J_{s,V}^2 \right] = \int_V  \int_V G(z,w) \, dz \, dw.\]

{The claim \eqref{eq:boundaryV} easily follows by applying either identity above to $\partial V$.}

We now prove \eqref{eq:JV}. Recall that
\[   \Cont_{\delta}(\A \cap V;s) =  e^{\eta s} \, \Vol\{z: \dist(z,\A \cap V)
    \leq e^{-s} \}.\]
In particular,
\[   \left|J_s -  \Cont_{\delta}(\A \cap V;s)\right|
 \leq  \int_{\dist(z, \partial V) \leq e^{-s} }
J_s(z) \, dz.\]
Therefore for $s$ sufficiently large,
\begin{equation}  \E\left[
|J_s -  \Cont_{\delta}(\A \cap V;s)|\right]
  \leq c \, \Vol\{z: \dist(z, \partial V) 
   \leq e^{-s}\} \leq c'  e^{-\epsilon s} 
\end{equation}
where $\epsilon$ comes from (\ref{eq:boundarycond}),  and hence by Markov's inequality and the Borel-Cantelli lemma, for all sufficiently large $s$,
\begin{equation}    |J_s -  \Cont_{\delta}(\A \cap V;s)| 
  \leq c \, e^{-\epsilon s/2}.\end{equation}
Therefore,
\[  \Cont_\delta(\A \cap V)
 = \lim_{s\rightarrow \infty}
  \Cont_{\delta}(\A \cap V;s) = \lim_{s \rightarrow
  \infty} J_s = J_V.\]
The last two claims of the proposition are easy corollaries of $\mu$-a.e.\ and $L^2$-convergence and dominated convergence theorem where the bounds are provided by \eqref{dec15.1}.
\end{proof}

\begin{proof}[Proof of Theorem \ref{theorem.dec16}] {Applying Proposition \ref{prop:generalMC} to the setup here, we have already proved the theorem for $V$ such that its closure $\overline{V}$ is disjoint from $\{0,\e\}$.} Hence it suffices to deal with the case where $V$ intersects 0 and/or $\e$ and show that in this case ${\rm Cont}_\delta (V\cap {\cal A})$ {exists and} is $\mu$-a.s.\ finite. Without loss of generality we may assume ${V= B(0,e^{-1})}$. For $n\in\N$ define
$
A_n=\{e^{-n-1}< |z| \leq e^{-n} \}
$
and {
$$
	K_n={\rm Cont}_\delta(A_n \cap {\cal A}),
$$	
which exists $\mu$-a.s.\ thanks to Proposition \ref{prop:generalMC}. As the boundaries of the $A_n$'s are all $(d-1)$-dimensional and ${\rm Cont}_{\delta}(\{0\})=0$, we see that
\begin{equation}\label{eq:annulus0}
{\rm Cont}_\delta (V\cap {\cal A})=\sum_{n=1}^{\infty} K_n,
\end{equation}
which also exists $\mu$-a.s. 
Hence the only thing left is to prove is that the right side of \eqref{eq:annulus0} is $\mu$-a.s.\ finite.}
 We will argue the existence of a $q>0$ such that
	\eqb
	\mu(K_n^2) \leq q^{-1}e^{-qn}.
	\label{eq102}
	\eqe
	This is sufficient to conclude the proof since it gives by Chebyshev's inequality 
	that $\mu(K_n\geq e^{-qn/2})\leq q^{-1}e^{-qn/2}$ for all sufficiently small $n$, so the Borel-Cantelli lemma implies that $K_n<e^{-qn/2}$ a.s.\ for all sufficiently large $n$ and we get $\sum_{n\in\N} K_n<\infty$ $\mu$-a.s.
	
	It remains to prove \eqref{eq102}. We write $\mu_{0,\mathbf e}$ instead of $\mu$ in the rest of the proof to emphasize the dependence of the measure on the end-points. By \eqref{eq:scalingpty} and the scaling property of the Minkowski content,
	\eqb
	\mu_{0,\mathbf e}(K_n^2) 
	= e^{(d-2)n} e^{- 2\delta n}\mu_{0,e^n \mathbf e}(K_1^2).
	\label{eq50}
	\eqe
	By the Brownian path decompositions in Section \ref{brownsec}, with $D$ the unit {disk or ball},
	\eqbn
	\mu_{0,e^n \mathbf e} = \int_{\partial D} \mu_{0,z}\oplus \mu^{D^c}_{z,e^n \mathbf e}\,\sigma(dz).
	\eqen
	For $d=2$, {by \eqref{eq:harnackresults},
\begin{equation}\label{eq:C1}
||\mu^{D^c}_{z,e^n \mathbf e}||\leq C.
\end{equation}}For $d=3,$ {by \eqref{eq:poissonbound},
\begin{equation}\label{eq:C2}
||\mu^{D^c}_{z,e^n \mathbf e}||\leq C||\mu_{z,e^n \mathbf e}||,
\end{equation}
and the right side is bounded by $ q^{-1}e^{-(d-2)n}$ for some sufficiently small $q>0$ by \eqref{eq:scalingpty}.} Therefore, decreasing $q$ if necessary, in either dimension the total mass of $\mu^{D^c}_{z,e^n \mathbf e}$ is bounded by $q^{-1}e^{-(d-2)n}$. Using this, rotation invariance property of $\mu_{0,z}$ for $z\in\partial D$, and the fact that adding an excursion from $\partial D$ to $e^n\mathbf e$ to a path from $0$ to $\partial D$ can only decrease the number of cut points in $D$, 
	\eqbn
	\mu_{0,e^n \mathbf e}(K_1^2) \leq q^{-1} e^{-(d-2)n} \mu_{0,\mathbf e}(K_1^2).
	\eqen
	Combining \eqref{eq50} and this estimate we get
	\eqbn
	\mu_{0,\mathbf e}(K_n^2) 
	= e^{(d-2- 2\delta)n} \mu_{0,e^n \mathbf e}(K_1^2)
	= q^{-1} e^{-2\delta n} \mu_{0,\mathbf e}(K_1^2),
	\eqen
	which implies \eqref{eq102} after possibly decreasing $q$ since $\mu_{0,\mathbf e}(K_1^2)<\infty$ a.s.
\end{proof}

\subsection{Path measure restricted to a bounded domain}\label{sec:bounded}
In this subsection, we will state a version of the main results of this work for the Brownian measure restricted to a bounded domain. We will also discuss how to adapt the proof to this case.

Let $U$ be an open set containing $0,\e$. Although morally we can {directly} deal with a much wider class of domains, for the same reason as stated at the beginning of Section \ref{brownsec}, we only consider the case of $U$ being a ball {or a half-space.} 
{Later, we will use conformal invariance to deal with other domains in two dimension.} 

We write $\mu^U=\mu^{U}_{0,\e}$ for short. We still define the cut-point Green's functions by (provided the limits exist)
\begin{equation}\label{eq:G1defU}
\greencut_U(z)=\lim_{s \rightarrow \infty
} \, \mu^U[J_s(z)] 
\end{equation}
and \begin{equation}\label{eq:G2defU}
\greencut_U(z,w)=\lim_{s \rightarrow \infty
} \, \mu^U\left[J_s(z) \, J_s(w)\right].
\end{equation}

We state variants of Theorems \ref{theorem.dec16.2} -- \ref{theorem.dec16} for $\mu^U$, and briefly explain how to modify the original proof to obtain the results in this subsection.
 
\begin{thm}  \label{thm1.1U}
There exists $u >0$ such that  if $e^{-b}  =
\dist(0, z, \e,\partial U)>0$, then
the limit in \eqref{eq:G1defU} 
exists, and if $s \geq b+1 $,
\begin{equation}\label{eq:T1c1U}
\mu^U\left[J_s(z)\right] 
    = \greencut_U(z) 
     \, [1+O(e^{(b-s)u})].
\end{equation}
     \end{thm}
Note that in this case we can no longer obtain the explicit form of $\greencut_U$ due to loss of symmetry.

To see that this version of Theorem \ref{theorem.dec16.2} (except \eqref{eq:T1c2}) still holds, we replace $\Gamma_{0,\e}$ by $\Gamma^{U}_{0,\e}$, $\mathbb{R}^d\setminus B$ by $U\setminus B$, etc. In particular, in the decomposition \eqref{eq:decompexample},  $\mu_{x,y}$ must be replaced by $\mu^U_{x,y}$. 
We first note that the versions of Lemmas \ref{dec8.lemma1}, \ref{feb1.1}, and \ref{dec8.lemma2} for $\mu^U$ follow automatically since $\mu^{U}$ is a restriction of $\mu$ and these results concern only upper bounds. Restricting to $U$ changes the definition of $\Phi_a$ (we write the new quantity as $\Phi_a^U$), resulting in that one cannot directly apply results from Section \ref{sec:3} as some exact scaling properties no longer hold.  To handle this, we see that by Lemma \ref{dec8.lemma2} and the new definition of $b$,  $|\Phi^U_a-\Phi_a|\leq e^{-ck}$ for some $c>0$ and note that Proposition \ref{prop102} still works in this setting if we set $V_0=U^c$ in \eqref{eq:quin} instead.
Hence, \eqref{eq:1.1final} still holds  with  $\Phi^U_a$  in place of $\Phi_a$. 

    We now turn to the two-point estimate.
    \begin{thm}\label{groundhogU}
     There exists $u> 0$
    such that  if $V \in \dyadic$, $V\subset U$, and $z,w \in V$, then with $e^{-b}=\dist(0,\e,z,w,\partial U)> 0$,
    the limit in  \eqref{eq:G2defU} exists, and, furthermore,   
   if $s \geq b+1$ and $0 \leq \rho \leq 1$, then 
\begin{equation*}
\mu^U\left[J_s(z) \, J_{s +\rho}(w)
   \right]  
    = \greencut_U(z,w)  \, 
      [1+O( e^{(b-s)u})].
\end{equation*}    
   Moreover,  there exists $0<c_{U,V}<C_{U,V} < \infty$ such that  
     \begin{equation*}
     c_{U,V} \, |z-w|^{-\eta}\leq \greencut_U(z,w) \leq C_{U,V} \, |z-w|^{-\eta}. 
\end{equation*}
    \end{thm}
   Theorem~\ref{groundhogU}  follows with similar modifications from the proof of  Theorem~\ref{groundhog}.
   
   Finally we turn to the existence of Minkowski content.
 \begin{thm} \label{thm1.3U}  Suppose  $d=2,3$ and $V$ is a {bounded Borel} subset of $\R^d$ such that $\partial V$ has zero $(d-\epsilon)$-Minkowski
  content for some $\epsilon > 0$.  Then, $\mu^U$-a.s.\
   the Minkowski
 content $ 
 \Cont_{\delta}(\A_\gamma \cap V)$ exists and 
 equals $J_V$.   Moreover,
 \[  \mu^U\left[\Cont_{\delta}(\A_\gamma \cap V)\right]
  =   \mu^U[J_V] = \int_V \greencut_U(z) \, dz.\]
 In particular, $\mu^U$-a.s., this holds for all   $V \in \dyadic$.
Moreover,
 if $V \in \dyadic$ and $V\subset U$, 
 \[ 
    \mu^U\left[\Cont_{\delta}(\A_\gamma \cap V)^2\right] =  \mu^U[J_V^2] =  \int_V \int_V \greencut_U(z,w) \, dz\, dw < \infty.\]
 \end{thm}
 
It is not difficult to see that the proof of Theorem \ref{theorem.dec16} can be adapted to the case of the restricted measure. Note in particular that \eqref{eq:T1c2} is not essential for the proof and that the argument in the last paragraph of the proof of Proposition \ref{prop:generalMC} still works with the above requirement on $V$. 
{
Once we obtain Theorem \ref{thm1.3U}, we can run an argument similar to that of Appendix \ref{se:ApA}, to conclude an analog of \eqref{eq:Borel0}, namely that one can induce a random non-atomic Borel measure supported on $\cA_\gamma$ that satisfies an analog of \eqref{eq:Borel}.

We now turn to general domains in two dimensions. Let $D$ be a simply connected planar domain and let $x,y$ be distinct points in $D$. One can find some disk or half-plane $U$ containing $0,\e$ such that there exists a conformal transformation $f: U \rightarrow D$  with $f(0) = x, f(\e) = y$.
Then we can define the path measure between 
$x$ and $y$ in $D$  to be the law of the image of the path measure in $U$ from $0$ to $\e$ (with time changed appropriately).
The cut points  of the image curve (denoted by $\cA_D$) are exactly the image of the cut points of the original curve.  
We therefore get the analog of Theorems \ref{thm1.1U} -- \ref{thm1.3U} for path measures in any domain. Note that $B(z,\epsilon)$ will not be exactly mapped to $B(f(z),f'(z)\epsilon)$, but as conformal maps 
can be locally approximated by the composition of a translation, a rotation, and a dilation, }
\eqref{eq:G1defU} and \eqref{eq:G2defU} still remain valid and we obtain the following relations:
The function $\greencut_D(\zeta ;z,w)$ satisfies the scaling rule
\begin{equation}\label{eq:scalingfirst}
 \greencut_D(z ;x,y)  = |(f^{-1})'(z)|^{5/4} \, 
 \greencut_{U}(f^{-1}(\zeta)).
\end{equation} 
Repeating the argument in Appendix \ref{se:ApA} we can argue that the Minkowski content defines a measure on the cut points and get analogs of \eqref{eq:Borel0} and \eqref{eq:Borel}.}
 \subsection{Half-plane excursions in $\C$}\label{excsec}
 \newcommand {\Half} {{\mathbb H}}
 \renewcommand {\Im}  {{\rm Im}}
 
 We have chosen to concentrate on the path measure $\mu_{0,\e}$ so far for convenience.  In \cite{HLLS}, we need the corresponding results for half-plane excursions in
 ${\mathbb C}$, which we state now. Since the existence of the limit defining Minkowski content is an almost sure property, it is immediate to see that this limit also exists for locally absolutely continuous measures.  A half-plane excursion  is a complex Brownian motion ``started at $0$ going to infinity staying in the upper half-plane $\Half$''. 
There are several ways to construct it, e.g.\ by considering the boundary-to-boundary measure defined in \eqref{eq:bdrytobdry} and then doing a conformal transformation and normalization. Alternatively, one can construct it	as an $h$-process with the harmonic function $h(z) = \Im(z)$.  
Let $X$ be a standard 1-dimensional Brownian motion and let $Y$ be an independent 3-dimensional Bessel process, both of which start from 0. 
Then $(X,Y)$ has the law of a half-plane excursion. We refer to~\cite{lawler-book} for more background on half-plane excursions.

  Here let us write $\Prob$ for the probability measure on half-plane excursions.  This is
 a probability measure on curves $\gamma:(0,\infty) \rightarrow \Half$ with 
 $\gamma(0^+) = 0$ and $\lim_{t\rta\infty}\gamma(t)=\infty$.
 This curve is transient, and hence we can talk about
 cut times and cut points.   
Given a compact set $U\subset \bbH$, on the event $\cA\cap U\neq \emptyset$, let $x,y\in\partial U$ be
 the first and last visit to $U$ by $\gamma$  respectively.
Then the segment of $\gamma$ between $x$ and $y$ has the law of  
$\mu_{x,y}$ conditioned to stay in $\bbH$. 

Let $\A\subset\Half$
denote the set of cut points.   As before, for $z\in \bbH$ we let
\[  J_{s}(z) = e^{5s/4} \, 1\{\dist(z,\A) \leq e^{-s}\}\quad\mbox{ and }\quad J_{s,V} = \int_V J_s(z) \, dz. \]

Similarly as Theorems~\ref{thm1.1U} -~\ref{thm1.3U}, we have the following variants of Theorems~\ref{theorem.dec16.2} - \ref{theorem.dec16}  in this setting.
 \begin{thm}\label{thm:excursion}
Suppose $V$ is a bounded Borel subset of $\Half$ such that
for some $\epsilon > 0$, the $(2-\epsilon)$-dimensional Minkowski content of $\partial
V$ is zero.  Then with probability one, the limit $J_V=\lim_{s \rightarrow \infty}J_{s,V}$ exists and gives the $(3/4)$-dimensional Minkowski content of $\A \cap V$.
In particular, with probability one, this holds simultaneously for every $V$ which is a dyadic square contained in $\Half$.

For $z,w\in\bbH$, $z\neq w$, the following limits exist 
\begin{equation}\label{eq:local}
\greencut_\bbH(z):=\lim_{s\to\infty}\E[J_s(z)]
\qquad\text{and}\qquad \greencut_\bbH(z,w):=\lim_{s\to\infty}\E[J_s(z)J_s(w)]
\end{equation}
For any compact set $V\subset\bbH$, there exists a constant $c_V>0$ such that $\greencut_\bbH(z,w) \leq c_V \, |z-w|^{-5/4}$ for $z,w\in V$ with $z\neq w$.
  
Moreover, $\{J_V\}$ extends to a non-atomic random Borel measure $m$ on $\Half$ such that for all Borel sets $U,U'\subset \bbH$,
\[
\E\left[m(U)\right] = \int_U \greencut_\bbH(z) \, dz
\quad\textrm{ and }\quad
\E\left[m(U)m(U')\right] = \iint_{U\times U'} \greencut_\bbH(z,w) \, dz\, dw.
\]
\end{thm}
For simplicity, in~\eqref{eq:local} we did not spell out the quantitative error term as in Theorems~\ref{thm1.1U} and~\ref{groundhogU}. However, similar quantitative statements can be easily formulated.
\begin{proof}[Proof of Theorem~\ref{thm:excursion}]
We only explain how to obtain the one-point estimate in \eqref{eq:local} from a slight variant of Theorem~\ref{thm1.1U}.
The rest of Theorem~\ref{thm:excursion} follows from more straightforward adaptations as in Theorems~\ref{groundhogU} and~\ref{thm1.3U}.

Consider a conformal map $\phi$ from $\D$ to $\bbH$ such that 
$\phi(-\e)=0$ and $\phi(\e)=\infty$. Then $\phi\circ\mu^{\D}_{-\e,\e}$ gives the law of a constant multiple of the half-plane excursion. 
Let $\greencut_\D(\cdot)$ be the cut-point Green's function in $\D$ from $-\e$ to $\e$. 
{By the discussion at the end of Section \ref{sec:bounded}, conformal invariance gives that
\begin{equation}\label{eq:scaling-rule}
\greencut_{\bbH}(\phi(\zeta))=| \phi'(\zeta)|^{5/4}\greencut_\D(\zeta ) /K
\end{equation}
where the normalizing constant
$K=\|\mu^{\D}_{0,\e}\|$.
Therefore  it suffices to prove the analog of   \eqref{eq:local} for $\mu^{\D}_{-\e,\e}$.} 

{Without loss of generality we only consider the case $z=0$. 
We now decompose $\omega$ under $\mu^{\D}_{-\e,\e}$ as $\omega^1\oplus\omega^2\oplus\omega^3$, where
\begin{itemize}
\item $\omega^1$ is $\omega$ stopped at the first entrance of $B(0,1/3)$;
\item $\omega^2$ is $\omega$ from the first entrance of $B(0,1/3)$ to the last exit of $B(0,1/2)$;
\item $\omega^3$ is $\omega$ from the last exit of $B(0,1/2)$ until the very end.
\end{itemize}
Conditioning on a pair of non-intersecting paths $\omega^1$ and $\omega^3$, the law of $\omega^2$ is governed by an interior-to-interior path measure in a disk. Hence we can rerun the argument in Section \ref{sec:bounded} and obtain a conditional version of Theorem \ref{thm1.1U},  treating $\D^c$, $\omega^1$ and $\omega^3$ as initial configurations when we apply separation lemmas or coupling results (e.g., to obtain an analog of \eqref{eq:1.1final}, in our case $V_0$, $V_1$, and $V_2$ should be set to $\D^c$, $\omega^1$, and $\omega^3$, respectively, in \eqref{eq:quin}) and consider non-intersection events (e.g., in an analog of \eqref{eq:gkdefinition}, one needs to take 
$\omega^1\oplus\gamma_1$ and $\gamma_2\oplus\omega^3$ instead of $\gamma_1$ and $\gamma_2$). Finally, integrating out $\omega^1$ and $\omega^3$ and noting that the error bounds are uniform  regardless of the initial configurations (see e.g. Proposition \ref{prop102}), we obtain the one-point estimate as desired.} 
\end{proof}

While we do not know the function $\greencut_\Half(\cdot)$ exactly, we can give  its asymptotics up to a multiplicative constant by using the exact value
of the Brownian half-plane intersection exponent. 

	\begin{prop}\label{prop:Green-half}
		The following hold for $x,y>0$ and $z\in\bbH$,
		\[     \greencut_\Half(x + i )   \asymp (x+1)^{-10/3} ,\]
		\[  \greencut_\Half (y(x+i)) = y^{-5/4} \, \greencut_\Half(x+i )  
		\asymp y^{-5/4} \, (x+1)^{-10/3}, \mbox{ and }\]
		\[  \greencut_\Half(z) \asymp [\Im(z)]^{-5/4} \, [\sin(\arg z)]^{10/3}.\]
	\end{prop}
	\begin{proof}

		The second claim follows from scaling of the first claim, and the third claim is a reformulation of the second claim. To see the the first claim,
		consider the M\"obius transform $\phi:\bbH\to\bbH$ given by $\phi(z)=\frac{z-|x+i|}{z+|x+i|}$ so that $\phi(0)=-\phi(-\infty)=-1$ and $\Re(y)=0$,
		where $y=\phi(x+i)$. See Figure \ref{fig:halfplane}.
			\begin{figure}
	\centering
	\includegraphics[scale=1]{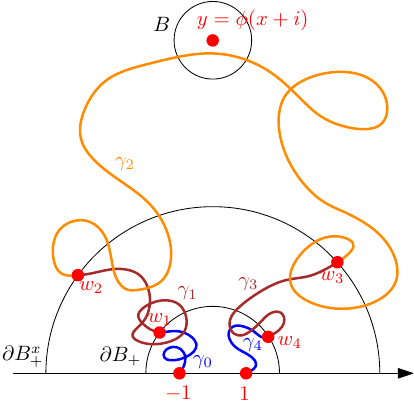}
	\caption{Illustration of the proof of Proposition \ref{prop:Green-half}.}
	\label{fig:halfplane}
\end{figure}
		Note that $\Im(y)\sim |\phi'(y)|\sim 2|x|$. By~\eqref{eq:scaling-rule}, the first claim is equivalent to
		$|x|^{5/4}\greencut_\bbH(y ;-1,1) \asymp   |x|^{-10/3}$. Since $0<|\mu_{-1,1}^{\bbH}|<\infty$,
		it is sufficient to show that  
		\begin{equation}\label{eq:new-estiamte}
		(xe^{-s})^{-\frac54}\mu_{-1,1}^{\bbH}(E^s) \asymp x^{-10/3}  \qquad \textrm{for } e^s>x>100,
		\end{equation}
		where $E^s$ is the set of paths from $-1$ to $1$ in $\bbH$ with a cut point in the ball $B=B(y,xe^{-s})$.
		
		Let $B_+=B(0,2)\cap \bbH$ and  $B^x_+=B(0,x) \cap \bbH$.
		Note that the path decomposition Lemmas~\ref{lem:decomp1}---\ref{lem:decomp3} still hold with $D=\{z\in \bbH: |z|>a \}$ and $S=\{z\in\bbH: |z|=b \}$ for some $b>a\ge 0$ since the exact same proofs as before work. Repetitively applying these lemmas as in Appendix~\ref{subsec:int-bdy}, we have  the following path decomposition
		\eqbn
		\begin{split}
			&\mu^\bbH_{-1,1} -\mu^{\bbH\setminus B}_{-1,1}\\
			&\,\,= \int \mu^{B_+}_{-1,w_1}  \oplus \mu^{B^x_+}_{w_1,w_2} \oplus (\mu^{\bbH}_{w_2,w_3}-\mu^{\bbH\setminus B}_{w_2,w_3}) \oplus \mu^{B^x_+}_{w_3,w_4}\oplus \mu^{B_+}_{w_4,1} \, \sigma(dw_1, dw_2,dw_3,dw_4),
		\end{split}
		\eqen
		where the integral is over $w_1,w_4\in \bdy B_+$ and $w_2,w_3\in \bdy B_+^x$.
		In other words, given a path $\gamma$ in the support of $\mu^\bbH_{-1,1} -\mu^{\bbH\setminus B}_{-1,1}$, which intersects $B$, we write it as $\gamma=\gamma_0\oplus\gamma_1\oplus\gamma_2\oplus\gamma_3\oplus\gamma_4$, where $w_1$ (resp., $w_4$) is the first (resp., last) visit of $\partial B_+\setminus\R$, and $w_2$ (resp., $w_3$) is the first (resp., last) visit of $\partial B_+^x\setminus\R$. Here for $i=0,1,2,3,4$ the end-points of $\gamma_i$ are $w_i$ and $w_{i+1}$ with the convention that $w_0=-1$ and $w_5=1$.
		
		Let $E^s_{w_2,w_3}$ be event that $\gamma_2$  has a cut point in $B$. By Observation~\ref{obs:upper},
		there exists $C>0$ not depending on  $e^s>x>100$ and $w_2,w_3\in \bdy B^x_+$ such that 
		\begin{equation}\label{eq:upper-cut}
		(xe^{-s})^{-\frac54}\mu_{w_2,w_3}^{\bbH}(E^s_{w_2,w_3})\le (xe^{-s})^{-\frac54}\mu_{w_2,w_3}(E^s_{w_2,w_3})\le C.
		\end{equation}
		On the other hand,  consider the measure
		$$\pi_{w_1,w_4}=\int_{{\bdy B^x_+ \cap \bbH}}    \mu^{B^x_+}_{w_1,w_2} \sigma(dw_2)  \times \int_{{\bdy B^x_+ \cap \bbH}} \mu^{B^x_+}_{w_4,w_3} \sigma(dw_3)$$ on the pair $(\gamma_1,\gamma^{\textrm R}_3)$.
		By \cite[equation (2)]{Law1998} and~\cite{LSWTrilogy1}, $\pi_{w_1,w_4}({\gamma_1\cap \gamma^{\textrm R}_3 = \emptyset}) \lesssim x^{-10/3}$ uniformly in $w_1,w_4$.  
		Combined with~\eqref{eq:upper-cut}, we get  
		\[
		(xe^{-s})^{-\frac54}\mu_{-1,1}^{\bbH}(E^s)  \lesssim |x|^{-10/3}\int_{\bdy B_+\times \bdy B_+} \|\mu^{B_+}_{-1, w_1}\| \|\mu^{B_+}_{w_4,1}\| \,\sigma(dw_1,dw_4)\lesssim |x|^{-10/3}.
		\]
		
		It remains to prove the lower bound in~\eqref{eq:new-estiamte}.
		Given a realization of $\gamma_0$ and $\gamma_4$, let $W(\gamma_0,\gamma_4)$ be  the set of realizations of $(\gamma_1,\gamma_3)$ such that 
		\[
		\gamma_1\cap \gamma_4=\gamma_3\cap \gamma_0=\emptyset,\quad  \dist(\gamma_1\cap A_x, \gamma_3\cap A_x)\ge 0.1x, \quad 
		{\textrm{and}\quad \gamma_1\cap \gamma_3=\emptyset,}
		\]
		where $A_x$ is the annulus $\{ z\in B_x: |z|>x/2\}$.
		
		Let $V_{0,4}=\{(\gamma_0,\gamma_4):  \dist(\gamma_0,\gamma_4)>0.1, \; \Im(w_1)\ge 0.1, \textrm{and }\Im(w_2)\ge 0.1 \}$.
		By the half-plane variant of the separation lemma,
		there exists $c_1>0$ such that $\pi_{w_1,w_4}({W(\gamma_0,\gamma_4)})\ge  c_1|x|^{-10/3}$ for all $(\gamma_0,\gamma_4)\in V_{0,4}$ and $x\ge 100$.
		On the other hand, let $E^s_{\gamma_1,\gamma_3}$ be the event that $\gamma_1\oplus\gamma_2\oplus\gamma_3$ has a cut point in $B$. 
		Then by a similar argument as how we argue the positivity of $c_*$ in \eqref{eq:cstardef} in Section \ref{sec:cstar}, there exists $c_2>0$ such that
		\begin{equation}\label{eq:lower-cut}
		(xe^{-s})^{-\frac54}\mu_{w_2,w_3}^{\bbH}(E^s_{\gamma_1,\gamma_3})\ge c_2
		\end{equation}
		if $\dist(\gamma_1\cap A_x, {\gamma_3\cap A_x})\ge 0.1x$  and $e^s>x>100$. Therefore
		\eqbn
		\begin{split}
			&(xe^{-s})^{-\frac54}\mu_{-1,1}^{\bbH}(E^s)\\
			&\qquad\ge c_1c_2 |x|^{-10/3}\int_{\bdy B_+\times \bdy B_+} \mu^{B_+}_{-1, w_1} \times \mu^{B_+}_{w_4,1}(\{(\gamma_0,\gamma_4)\in V_{0,4} \} )\,  \sigma(dw_1,dw_4),
		\end{split}
		\eqen
  which is $\gtrsim |x|^{-10/3}$. This gives the desired lower bound in~\eqref{eq:new-estiamte}.
	\end{proof}

\appendix
{\section{Minkowski content induces a measure}\label{se:ApA}}
In this appendix, we discuss how to induce a measure from the Minkowski content of $\cA$, and {prove claims \eqref{eq:Borel0} and \eqref{eq:Borel} in Section \ref{sec:intro}.

Keeping the same notation as in Section \ref{sec:intro}, }
we first define $\wt\cD_n$ as the union of the empty set and sets $V$ of the form \eqref{eq:dyadicV} with  $\left(\frac{k_i}{2^n},\frac{k_i+1}{2^n}\right]$ replaced by $\left(\frac{k_i}{2^n},\frac{K_i}{2^n}\right]$, $\left(\frac{k_i}{2^n},\infty
\right)$, or $\left(-\infty,\frac{k_i}{2^n}\right]$, $k_i,K_i\in\mathbb{Z}$ for $i=1,\dots,d$ {and then write $\wt\cD=\cup_{n\in\mathbb{Z}^+ }\wt\cD_n$}. We observe that $\wt\cD$ forms a semialgebra. We will use  standard measure-theoretic results (see e.g.\ \cite[Theorem A.1.1]{Durrett}) to argue that $\nu$ defined below can be uniquely extended  
to a random Borel measure $\mu$-a.s.

{Recall the definition of $\cD$ in \eqref{eq:dyadico}. For $V\in\cD$, set $\nu(V) := \Cont_{\delta}(\A_\gamma \cap V) $ and for $V\in \wt\cD\setminus\cD$, we will define $\nu(V)$ as follows:
\begin{itemize}
\item If $V\in \cD_n$,
set
\begin{equation}\label{eq:semialgebra}
\nu(V): = \sum_{j\in J}\nu(V_j)
\end{equation}
with  $V_j\in\dyadic$, $j\in J$ a finite (if $0,\e\notin\overline{V})$ or countable (if $\overline{V}$ contains $0$ or $\e$) partition of $V$. However, in the latter case, we restrict to partitions such that for any $\iota>0$, the number of cubes in $J$ not contained in $B(0,\iota)$ or $B(\e,\iota)$ is finite so that for some $\epsilon>0$,
\begin{equation}\label{eq:nocontent}
 \Cont_{d-\epsilon}(\cup_{j\in J}\partial V_j)=0.
 \end{equation}
\item If 
\begin{equation}\label{eq:VDn}
V\in \wt\cD_n\setminus \cD_n,
\end{equation}
set $\nu(V)$ similarly as in \eqref{eq:semialgebra}, but with $V_j\in \cD_n$. 
\end{itemize}}
In any of the cases, if the right side of \eqref{eq:semialgebra} is infinite we also set $\nu(V)$ to be infinite, although we expect $\Cont{_\delta}(\A)$ to be $\mu$-a.s.\ finite by an argument similar to that in the proof of Theorem \ref{theorem.dec16} at the end of Section \ref{gensec}. {By Theorem \ref{theorem.dec16} and \eqref{eq:nocontent}, and the fact that $\A$ is {compact} $\mu$-a.s., we have  $$\Cont_{\delta}(\cup_{j\in J}\partial V_j\cap\A)=0,\quad \mbox{ $\mu$-a.s.}$$ for any partition $J$ as above. {In light of this, and also recalling} the definition of Minkowski content, one can argue that $\nu(V)$ is well-defined $\mu$-a.s., i.e.\ independent of the choice of partition.}
{One can also check} conditions (i) and (ii) in Theorem A.1.1 of \cite{Durrett} are both satisfied. 
Hence $\nu$ can be uniquely extended to a (random) $\sigma$-finite compactly supported  Borel measure on $\R^d$ $\mu$-a.s. {In fact, it is not difficult to see that $\nu$ is $\mu$-a.s.\ supported on $\cA$.}
Moreover, for all Borel sets $U,U'$,
\[
  \mu\left[\nu(U)\right] = \int_U G^{\rm cut}(z) \, dz\qquad\textrm{and}\qquad \mu\left[\nu(U)\nu(U')\right] = \iint_{U\times U'} G^{\rm cut}(z,w) \, dz\, dw.
\]

Restricted to a fixed $V\in \cD$, {by \eqref{eq:T1c2} and \eqref{eq:T2c2}, one has $\mu[\nu(U)]\asymp {\rm Vol}(U)$  and $\mu[\nu(U)^2]\asymp{\rm Vol}(U)^2$, hence $\mu[\nu(U)^2]\asymp \mu[\nu(U)]^2$}  for all Borel $U\subset V$. {We now claim that} $\nu$ is non-atomic on $V$: this follows by fixing an arbitrary $\eps>0$ and using a union bound and Chebyshev's inequality to prove that for all $\delta$-sized squares in a bounded region the probability that the $\nu$-mass of some square is larger than $\eps$ goes to zero as $\delta\rta 0$.
Varying $V$, we see that $\nu$ is non-atomic on $\R^d$.  
\section{Brownian path decomposition}
\label{appendix}
In this appendix we provide more details on the Brownian path decomposition in Section~\ref{brownsec}. 
Throughout this section we assume $d\in \{2,3\}$ and $D\subset\R^d$ is such that $\partial D$, if nonempty,  is a disjoint union of {lines and circles (for $d=2$) or planes and spheres (for $d=3$)}. 
\subsection{Brownian path measure from conditioning}\label{subsec:interior}
Suppose $D\neq \R^{d}$ and $x,y\in D$.  We first explain that $\mu^{D}_{x,y}$ can be obtained from Brownian motion by a conditioning. Let $S_\eps(y):=\{ \zeta\in \R^d: |\zeta-y|=\eps  \}$. 
Choose $\eps>0$ small enough such that $S_\eps(y)\subset D$ and $\eps<|x-y|$. 
Let $\mu^D_{x,y,\eps}$ denote the 
measure of a Brownian motion in $\R^d$ started from $x$ stopped
when it reaches $S_\eps(y)$, restricted to the event $E_{x,y,\eps}$
that the Brownian motion reaches $S_\eps(y)$ before leaving $D$.  
Let $G_2(x) = - \frac 1{\pi} \, \log |x|$ for $x\in\R^2$ and $G_3(x)=  \frac{1}{2\pi\, |x|}$ for $x\in \R^3$. 
\begin{lemma}\label{lem:lim}
Suppose $D\neq \R^{d}$ and $x,y\in D$. Then 
\(\mu^{D}_{x,y} = \lim_{\eps\to 0} G_d(\eps) \mu^D_{x,y,\eps}\) 
\end{lemma}
\begin{proof}
The probability measure $\mu^{D}_{x,y} /\|\mu^{D}_{x,y}\|$ is an example of a so-called  \emph{conditioned Brownian motion} in the sense of Doob; see \cite[Section~6]{Doob1957}. 
In particular,  using Doob-h transform,  we see that  $ \mu^D_{x,y,\eps}/\|\mu^D_{x,y,\eps}\|$ weakly converge to $\mu^{D}_{x,y} /\|\mu^{D}_{x,y}\|$. Now proving Lemma~\ref{lem:lim} reduces to showing that 
$\|\mu^{D}_{x,y}\| = \lim_{\eps\to 0} G_d(0,\eps) \|\mu^D_{x,y,\eps}\|$.
Let $D_\eps$ be the component of $D\setminus S_\eps$ containing $x$. For fixed $y$ and $\eps$, the function $x\mapsto\|\mu^D_{x,y,\eps}\|$ is the unique harmonic function on $D_\eps$ with boundary value $1$ on $S_\eps(y)$ and $0$ on $\partial D$. Since $\|\mu^{D}_{x,y}\|=G_D(x,y)$, showing $\|\mu^{D}_{x,y}\| = \lim_{\eps\to 0} G_d(\eps) \|\mu^D_{x,y,\eps}\|$ is an easy exercise in harmonic function, which we leave to the reader.
\end{proof}

\subsection{First and last passage decompositions}
\label{appendix2}
Suppose $x,y\in D\cup \partial D$ and  $S\subset D$ is a {line or circle (for $d=2$) or plane or sphere (for $d=3$)} such that $x\notin S$ and $y\notin S$.  Let $D_x$ (resp., $D_y$) denote the component  of $D\setminus S$ whose closure  contains $x$  (resp., $y$). We make the convention that   $\mu_{x,y}^{D_y} =\mu_{x,y}^{D_x}=0$  if $D_x\neq D_y$.
\begin{lemma}[First passage decomposition]\label{lem:decomp1}
\(\mu^{D}_{x,y} =  \mu_{x,y}^{D_x} +\int_S [\mu_{x,\zeta}^{D_x} \oplus \mu_{\zeta,y}^D] 
\, \sigma(d\zeta)\).
\end{lemma}
\begin{proof}
We first assume that $D\neq \R^{d}$ and $x,y\in D$.
Choose $\eps>0$ small enough such that $S_\eps(y)\cap S=\emptyset$.
Let $\mu_{x,y,\eps}^{D_x}$ and $\mu_{\zeta,y,\eps}^D$ be defined in the same way as $\mu^D_{x,y,\eps}$.
By the strong Markov property of Brownian motion, we have that 
\begin{equation}\label{eq:eps-decomp}
 \mu^{D}_{x,y,\eps} =  \mu_{x,y,\eps}^{D_x} +\int_S [\mu_{x,\zeta}^{D_x} \oplus \mu_{\zeta,y,\eps}^D] 
 \, \sigma(d\zeta).
\end{equation}
Here we use the  fact that  $\int_S\mu_{x,\zeta}^{D_x}\, \sigma(d\zeta)$  
equals the
measure of a Brownian motion started from $x$ stopped
when it reaches $S$, restricted to the event that the path reaches $S$ before leaving $D$; 
see \eqref{eq:int-bdy}.
Moreover, let $\mu^{D}_{x,S,y,\eps}$ be the measure on  paths drawn from $\mu^{D}_{x,y,\eps} -\mu_{x,y,\eps}^{D_x}$ stopped at hitting $S$. 
Then  
\(\mu^{D}_{x,S,y,\eps}=  \int_S \mu_{x,\zeta}^{D_x} \cdot  \|\mu_{\zeta,y,\eps}^D\| \, \sigma(d\zeta)\) 
by Bayes rule. 

Now this case of Lemma~\ref{lem:decomp1} follows from Lemma~\ref{lem:lim} by multiplying by $G_d(\eps)$ and taking the limit as $\eps\rta 0$ on both sides of 	equation~\eqref{eq:eps-decomp}.

We now assume that $D\neq \R^{d}$, $x\in D$ and $y\in \partial D$. Let $\n_y$ be
the inward unit normal at $y$ into $D$. 
Then Lemma~\ref{lem:decomp1} holds  with $y_\eps=y+ \eps\n_y$ in place of  $y$.  
Sending $\eps\rta 0$ after renormalization, we get the desired statement for $y$.

If $D\neq \R^{d}$, $x\in \partial D$ and $y\in \partial D$, then Lemma~\ref{lem:decomp1} holds  with $x_\eps=x+ \eps\n_x$ in place of  $x$.
Again by sending $\eps\rta 0$, we get the desired statement for $x$.

The statement for $D=\R^d$ and $x,y\in \R^d$ can be obtained by  taking the $R\to\infty$ limit for the case $D=\{z\in \R^d: |z|<R \}$.
\end{proof}
\begin{lemma}[Last passage decomposition]\label{lem:decomp2}
 \( \mu^{D}_{x,y} =  \mu_{x,y}^{D_x} + \int_{S} [\mu_{x,\zeta}^D \oplus \mu_{\zeta,y}^{D_y}
]\, \sigma(d\zeta)\).
\end{lemma}
\begin{proof}
Switching $x,y$ in Lemma~\ref{lem:decomp1} we get \(\mu^{D}_{y,x} =  \mu_{y,x}^{D_y} +\int_S [\mu_{y,\zeta}^{D_y} \oplus \mu_{\zeta,x}^D] 
 \, \sigma(d\zeta)\). Reversing time in this equation, we have 
 \(\mu^{D}_{x,y} =  \mu_{x,y}^{D_y} +\int_S [\mu_{x,\zeta}^D \oplus \mu_{\zeta,y}^{D_y}]
 \, \sigma(d\zeta)\).
 Note that $\mu_{x,y}^{D_y} =\mu_{x,y}^{D_x} $, which is $0$ if $D_x\neq D_y$. This gives Lemma~\ref{lem:decomp2}. 
\end{proof}
We now deduce the case when one of the end-points is on $S$, which is either a {line or circle (for $d=2$) or plane or sphere (for $d=3$)}.
\begin{lemma}\label{lem:decomp3}
If $z\in S$, then
\[
\mu^{D}_{x,z} =\int_{S} [\mu_{x,\zeta}^{D_x} \oplus \mu_{\zeta,z}^D]
\, \sigma(d\zeta)\quad\textrm{and}\quad \mu^{D}_{z,x} =\int_{S} [\mu_{z,\zeta}^D\oplus  \mu_{\zeta, x}^{D_x} ]
\, \sigma(d\zeta).
\]
\end{lemma}
\begin{proof}
Assume $D_x\neq D_y$ in Lemma~\ref{lem:decomp1} so that $\mu_{x,y}^{D_x} =0$. 
Sending $y$ to $z$, we get the first identity. The second one is given by the time reversal of the first one.
\end{proof}

\subsection{Proof path decompositions in Section~\ref{brownsec}}\label{subsec:int-bdy} 
All of the path decomposition identities 
in Section~\ref{brownsec} (namely, \eqref{eq:decomp11}---\eqref{jan22.1})  can be obtained by repeatedly applying Lemmas~\ref{lem:decomp1}---\ref{lem:decomp3}. 

We start by proving~\eqref{eq:decomp11}---\eqref{eq:decomp13}. The first two equations follow  from Lemmas~\ref{lem:decomp1} and~\ref{lem:decomp2}, respectively. 
Note that  \(\mu^{D}_{x,\zeta_2} =\int_S   [ \mu_{x,\zeta_1}^{D_1} \oplus \mu_{\zeta_1,\zeta_2}^D]  \, \sigma(d\zeta_1)\) by Lemma~\ref{lem:decomp3}.
Plugging this to~\eqref{eq:decomp12} yields~\eqref{eq:decomp13}.

We now prove~\eqref{eq:decomp21} and~\eqref{eq:decomp22}. First of all,
by Lemma~\ref{lem:decomp3} with $D=\R^d$, we have \(\mu_{x,y}=\int_{S_2}[\mu^{D_1}_{x,\zeta_2}  \oplus
\mu_{\zeta_2,y} ] \,  \sigma(d\zeta_2)\). By Lemma~\ref{lem:decomp1} we have \(\mu^{D_1}_{x,\zeta_2} = 
\int_{S_1} [\mu_{x,\zeta_1} ^{D_1}\oplus \mu^{D}_{\zeta^1,\zeta^2}  ] \,  \sigma(d\zeta_1)\).
This give~\eqref{eq:decomp21}. Again by  Lemma~\ref{lem:decomp3},  we have \(\mu_{x,y} =\int_{S_1} 
[\mu_{x,\zeta_1} \oplus \mu^{D_2}_{\zeta_1, y} ] \,  \sigma(d\zeta_1)\). 
By Lemma~\ref{lem:decomp1}, we have
\( \mu^{D_2}_{\zeta_1, y}=   \int_{S_2}[\mu^{D}_{\zeta^1,\zeta^2} \oplus
\mu_{\zeta_2,y}^{D_2} ] \,  \sigma(d\zeta_2)\). This gives~\eqref{eq:decomp22}.

To prove~\eqref{eq:decomp3} and~\eqref{jan22.1}, first note that \(  \mu^D_{\zeta_1,\zeta_2} =  \int_{S_4}
[\mu^{D'}_{\zeta_1,\zeta_4}  \oplus \mu_{\zeta_4,\zeta_2}^D ] \,  \sigma(d\zeta_4)\) by Lemma~\ref{lem:decomp1}. Since  \(  \mu^{D'}_{\zeta_1,\zeta_4} = \int_{S_3}
[\mu_{\zeta_1,\zeta_3} ^{D'}\oplus \mu^{D_0}_{\zeta^3,\zeta^4} ] \,  \sigma(d\zeta_3)\)
by Lemma~\ref{lem:decomp2}, we get~\eqref{eq:decomp3}. Similarly, we have 
\( \mu^D_{\zeta_1,\zeta_4} = \int_{S_3} 
[\mu_{\zeta_1,\zeta_3} ^{D_3}\oplus \mu^{D }_{\zeta^3,\zeta^4}  ] \,  \sigma(d\zeta_3)\) by Lemma~\ref{lem:decomp1}, and 
\( \mu^D_{\zeta_1,\zeta_2} = \int_{S_4}
[\mu^D_{\zeta_1,\zeta_4}  \oplus\mu_{\zeta_4,\zeta_2}^{D_4} ] \,  \sigma(d\zeta_4)\) by Lemma~\ref{lem:decomp2}. This gives~\eqref{jan22.1}.

{
\section{Some Poisson kernel estimates}
In this appendix, we prove the Poisson kernel estimates \eqref{eq:C1} and \eqref{eq:C2} (\eqref{eq:harnackresults} and \eqref{eq:poissonbound} in this appendix) that are used in the proof of Theorem \ref{theorem.dec16}. 

Consider $d=2$ and let $D$ be the unit disk. Then there exists some $C<\infty$ such that for any $y\in\partial D$ and  $|z|\geq 2$, 
\begin{equation}\label{eq:harnackresults}
\| \mu^{D^c}_{y,z}\|\leq C.
\end{equation}
To see this, recall that $\mu^{D^c}_{y,z}$ is the limit of $\frac{1}{2\epsilon} \mu^{D^c}_{y+\epsilon {\bf n}_y,z}$ and that the Green's function is a conformally invariant quantity.
Consider the conformal map that takes $D$ to $D^c$, $y$ to $y$, and $0$ to $z$. As the norm of the derivative of this conformal map at $z$ is uniformly of order 1,  \eqref{eq:harnackresults} follows from \eqref{eq:poissonestimates}.

We now consider $d=3$ and still let $D$ denote the unit ball. We claim the following bound: there exists $C<\infty$, such that for any $y\in\partial D$ and  $|z|\geq 2$,
\begin{equation}\label{eq:poissonbound}
\| \mu^{D^c}_{y,z}\|\leq C\|\mu_{y,z}\|.
\end{equation}
To see this, we decompose both measure as follows, with $D'=B(1.5)$:
$$
\mu^{D^c}_{y,z}=\int_{x\in\partial D'} \mu^{D'\setminus D}_{y,x}\oplus \mu^{D^c}_{x,z}\sigma(dx)
$$
and
$$
\mu_{y,z}=\int_{x\in\partial D'} \mu^{D'}_{y,x}\oplus \mu_{x,z}\sigma(dx),
$$
and observe that $\|\mu^{D^c}_{x,z}\|\leq \|\mu_{x,z}\|$.
Then the claim follows from the fact that there exists some $C<\infty$ such that
$$
\int_{x\in\partial D'} \|\mu^{D'\setminus D}_{y,x}\|\sigma(dx)=C
$$
and that
$$
 \int_{x\in\partial D'} \|\mu^{D'}_{y,x}\| \sigma(dx)=1.
$$
}

\bibliographystyle{hmralphaabbrv}

\end{document}